\documentclass[11pt, a4paper,reqno]{amsart}
\usepackage{a4,amsmath,amssymb,epic}
\usepackage{graphicx}
\usepackage{multicol}
\usepackage{rotating}
\usepackage{lscape}
\usepackage{amssymb,amsthm}
\usepackage{enumitem}
\usepackage[all]{xy}
\usepackage{tikz}
\usepackage{young}
\usepackage[vcentermath,enableskew]{youngtab}

\allowdisplaybreaks
\theoremstyle{plain}
\newtheorem{lem}{Lemma}[section]
\newtheorem{prop}[lem]{Proposition}
\newtheorem{thm}[lem]{Theorem}
\newtheorem{cor}[lem]{Corollary}

\newtheorem{defn}[lem]{Definition}

\newtheorem{example}[lem]{Example}

\newcommand{\ind}{\operatorname{ind}}

\newcommand{\half}{\frac{1}{2}}

\newcommand{\bbf}{\mathbb{F}}
\newcommand{\catmod}{\rm \textbf{mod}}

\newcommand{\sym}{\mathfrak{S}}

\begin{document}
\title[The  partition algebras in positive characteristic]{The blocks  of the partition algebra in positive characteristic }
\author{C.~Bowman}
  \email{chris.bowman.2@city.ac.uk}
 \author{M.~De Visscher}
  \email{Maud.Devisscher.1@city.ac.uk }
  \author{O.~King}
   \email{Oliver.King.1@city.ac.uk }
\address{Department of Mathematics,
 City University London,
 Northampton Square,
 London,
 EC1V 0HB,
 England.}

\date{\today}

 \maketitle

\begin{abstract}
In this paper we describe the blocks of the partition algebra over a field of positive characteristic.
\end{abstract}

\section*{Introduction}

The partition algebra $P_n^\mathbb{C}(\delta)$ was originally defined by P. Martin in \cite{MR1103994} over the complex field $\mathbb{C}$ as a generalisation of the Temperley-Lieb algebra for $\delta$-state $n$-site Potts models in statistical mechanics. Although this interpretation requires $\delta$ to be integral, it is possible to define the partition algebra $P_n^{\mathbb{F}}(\delta)$ over any field $\mathbb{F}$ and for any $\delta\in \mathbb{F}$. 
In a subsequent paper \cite{martin1996structure} P.~Martin investigates the representation theory of $P_n^\mathbb{C}(\delta)$. In particular, he showed that $P_n^\mathbb{C}(\delta)$ is semisimple unless $\delta \in \{0,1,\ldots , 2n-2\}$. Moreover, in the non-semisimple case (and when $\delta \neq 0$), he gave a complete description of the blocks of $P_n^\mathbb{C}(\delta)$. 

The aim of this paper is to describe the blocks of the partition algebra $P_n^\bbf(\delta)$ over a field $\bbf$ of positive characteristic when $\delta \neq 0$. Very little investigation has been made into the representation theory of the partition algebra in positive characteristic. It was shown by C. Xi in \cite{xi1999partition} that, for an arbitrary field $\bbf$ and an arbitrary parameter $\delta\in \bbf$, the partition algebra $P_n^\bbf(\delta)$ is cellular (as defined in \cite{Graham1996cellular}). We will use this cellular structure to prove our results.

The main result of this paper is given in Theorem 8.9.
It gives a description of the blocks of the partition algebra $P_n^\bbf(\delta)$, when $\delta$ lies in the prime subfield of $\bbf$, in terms of the action of an affine reflection group of type $A$. Similar results have been obtained for other diagram algebras.  In \cite{cox2009geometric}, A. Cox, M. De Visscher and P. Martin gave a description of the blocks of the Brauer algebra over the complex field in terms of the action of a reflection group of type $D$. In the same paper, they showed that the action of the corresponding affine reflection group provides a necessary condition for the blocks of the Brauer algebra over a field of positive characteristic. They note however that this is not a sufficient condition. Similar results have also been obtained for the walled Brauer algebra in \cite{cox2008walled}.
For the partition algebra, we show here that the description of the blocks in characteristic zero given by Martin can be rephrased in terms of the action of a reflection group of type $A$. Then, by contrast to the Brauer algebra case, the action of the corresponding affine reflection group provides a necessary and sufficient condition for the blocks of the partition algebra in positive characteristic.

The paper is organised as follows. Sections 1--5 are mainly expository. The new results are contained in Sections 6--9. In Section 1, we give a brief review of the representation theory of cellular algebras. In particular, we explain how the so-called \lq cell-blocks' determine the blocks of the algebra, justifying the focus on cell-blocks in this paper.
In Section 2, we give all the necessary combinatorics of partitions needed for the paper. Section 3 recalls the main results on the modular representation theory of the symmetric group. In Section 4 we give the definition of the partition algebra and the half partition algebras (which will be used in our proofs) and review the construction and properties of their cell modules, following \cite{martin1996structure} and \cite{xi1999partition}.  In Section 5 we review the main results on the representation theory of the partition algebra over a field of characteristic zero as developed in \cite{martin1996structure}. We conclude this section by rephrasing Martin's description of the blocks in terms of the action of a reflection group of type $A$ (see Theorem 5.4).

In Sections 6-9 we study the blocks of the partition algebra $P_n^\bbf(\delta)$ over a field $\bbf$ of positive characteristic when $\delta\in \bbf$ is non-zero. In Section 6, we use the action of the Jucys-Murphy elements on the cell modules to find a necessary condition for the blocks (Corollary 6.7). We then split into two cases, depending on the parameter $\delta$. In Section 7 we assume that  $\delta$ is not in the prime subfield $\mathbb{F}_p\subset \bbf$. In this case, we can easily deduce from the necessary condition from Section 6 and known results connecting the representation theory of the symmetric groups and that of the partition algebra that the blocks are simply given by blocks of the corresponding  symmetric group algebras (Theorem 7.2). In Section 8, we assume that $\delta \in \bbf_p$ (and $\delta \neq 0$). This case is more complicated. We introduce the notion of a $\delta$-marked abacus associated to each partition, and use the modular representation theory of the symmetric groups, together with the ordinary representation theory of the partition algebra and \lq reduction modulo $p$' arguments  to obtain a combinatorial description of the blocks of the partition algebra in this case (Theorem 8.8). We then reformulate this result in terms of the action of an affine reflection group of type $A$ in Theorem 8.9.
In the final section, Section 9, we define the notion of limiting blocks for the partition algebra, coming from certain full embeddings of categories of $P_n^\bbf(\delta)$-modules  as $n$ increases. We show that, surprisingly, we can give a proof for the limiting blocks (when $\delta \in \bbf_p$) which does not use the modular representation theory of the symmetric group.

\vspace{1cm}

\noindent \textbf{\large{Notations:}} The following notations will be used throughout the paper, except for Section 1.
 We fix a prime number $p$ and  a $p$-modular system $(K,R,k)$, that is, $R$ is a discrete valuation ring with maximal ideal $\mathfrak{m}=(\pi)$, $K={\rm Frac}(R)$ is its field of fractions (of characteristic zero) and $k=R/\mathfrak{m}$ is the residue field of characteristic $p$. We assume that $K$ and $k$ are algebraically closed. We will identify $\mathbb{Z}$ with $\mathbb{Z}1_R$ in $R$. This will allow us to consider elements of $\mathbb{Z}$ in $K$ via the embedding $R\hookrightarrow K$, and in $k$ via the projection $R\twoheadrightarrow k$. We will use $\mathbb{F}$ to denote either $K$ or $k$.

\section{Representation theory of cellular algebras: a short review}

J. Graham and G. Lehrer defined a new class of algebras, called cellular algebras, in \cite{Graham1996cellular}, and developed their representation theory. The symmetric group algebra and the partition algebra are examples of such algebras.

We will not need the precise definition of a cellular algebra (see \cite[(1.1)]{Graham1996cellular}) but will recall the main properties needed for this paper. All the details can be found in \cite{Graham1996cellular} or \cite[Chapter 2]{Mathas1999Hecke}.

Let $A$ be an associative algebra over a commutative ring $R$ with identity. If the algebra $A$ is cellular then it comes with a  basis satisfying certain multiplicative properties given in terms of a poset $(\Lambda, \leq)$. For each $\lambda\in \Lambda$ we have an $A$-module $W(\lambda)$, called the cell module corresponding to $\lambda$. Moreover, there is an $R$-bilinear form $\phi_\lambda$ on each $W(\lambda)$. 

Now assume that $R$ is a field and that $A$ is finite dimensional. Let $\Lambda_0$ be the subset of $\Lambda$ consisting of all $\lambda\in \Lambda$ such that $\phi_\lambda \neq 0$. Then for all $\lambda\in \Lambda_0$ the cell module $W(\lambda)$ has a unique simple quotient $L_\lambda$ (given as the quotient of $W(\lambda)$ by the radical of $\phi_\lambda$). Moreover, the set
$$\{L_\lambda \,\, : \,\, \lambda\in \Lambda_0\}$$
gives a complete set of non-isomorphic simple $A$-modules (see \cite[(3.4)]{Graham1996cellular}).

For each $\lambda\in \Lambda$ and $\mu\in \Lambda_0$ we denote by $d_{\lambda \mu}$ the multiplicity of $L_\mu$ as a composition factor of $W(\lambda)$. The matrix $D=(d_{\lambda\mu})_{\lambda\in\Lambda, \, \mu\in \Lambda_0}$ is called the decomposition matrix of $A$.

For each $\lambda\in \Lambda_0$ we denote by $P_\lambda$ the projective indecomposable $A$-module corresponding to $\lambda$, that is the projective cover of $L_\lambda$. For $\lambda, \mu\in \Lambda_0$ we denote by $c_{\lambda \mu}$ the multiplicity of $L_\mu$ as a composition factor of $P_\lambda$. The matrix $C=(c_{\lambda\mu})_{\lambda,\mu\in\Lambda_0}$ is known as the Cartan matrix of $A$.

\begin{thm}\label{thmcelldec} \cite[(3.6) and (3.7)]{Graham1996cellular} 

(i) The matrix $D$ is unitriangular; i.e. $d_{\lambda \mu}=0$ unless $\lambda\leq \mu$, and $d_{\lambda \lambda}=1$.

(ii) We have $C=D^t D$.
\end{thm}

 Now, assume for a moment that $A$ is any finite dimensional algebra (not necessarily cellular) with $\Lambda_0$ indexing its simple modules. Then $A$ decomposes uniquely as a direct sum of indecomposable 2-sided ideals
\begin{equation}\label{blockidempotent}
A=e_1A\oplus e_2A \oplus \ldots \oplus e_tA
\end{equation}
where $1=e_1+e_2+\ldots + e_t$ is a decomposition of $1$ as a sum of primitive central idempotents. The direct summands in (\ref{blockidempotent}) are called the blocks of $A$. We say that an $A$-module, $M$,  belongs to the block $e_iA$ if $e_iM=M$ and $e_jM=0$ for all $j\neq i$. In particular a simple module always belongs to a block, and an arbitrary module $M$ belongs to a block, $e_iA$ say, if and only if all of its composition factors belong to $e_iA$. So we can describe the blocks of $A$ by describing the labellings $\lambda\in \Lambda_0$ of all simple modules belonging to the same block. It is well known that these can be obtained from the Cartan matrix of $A$ as follows. Let $\lambda,\mu\in \Lambda_0$. We say that $\lambda, \mu$ are linked if $c_{\lambda \mu}\neq 0$. Then the \emph{blocks} of $A$ are given by the equivalence classes of the equivalence relation on $\Lambda_0$ generated by this linkage.

If we now go back to our assumption that the algebra $A$ is cellular, we have a further equivalent description of the blocks via cell modules which we now recall. Let $\lambda\in\Lambda$ and $\mu\in\Lambda_0$. We say that $\lambda,\mu$ are cell-linked if $d_{\lambda \mu}\neq 0$. The equivalence classes of the equivalence relation on $\Lambda$ generated by this cell-linkage are called the \emph{cell-blocks} of $A$. It follows from Theorem \ref{thmcelldec} that each block of $A$ is given as the intersection of a cell-block with $\Lambda_0$ (see \cite[(3.9.8)]{Graham1996cellular}). Thus for a cellular algebra, the problem of determining the blocks is equivalent to determining the cell-blocks. For that reason, we will be focussing on cell-blocks in this paper. 

\begin{defn} Let $A$ be any finite dimensional cellular algebra over a field and denote by $\Lambda$ the poset indexing the cell $A$-modules. For each $\lambda\in \Lambda$ we define $\mathcal{B}_\lambda(A)\subseteq \Lambda$ to be the cell-block of $A$ containing $\lambda$. 
\end{defn}

\section{Combinatorics of partitions}

\subsection*{Partitions and Young diagrams}\label{sec:partitions}

Given a natural number $n$, we define a partition $\lambda=(\lambda_1,\lambda_2,...)$ of $n$ to be a weakly decreasing sequence of non-negative integers such that $\sum_{i>0}\lambda_i=n$.
As we have $\lambda_i=0$ for $i\gg0$  we will often truncate the sequence and write $\lambda=(\lambda_1,\dots,\lambda_l)$, where $\lambda_l\neq0$ and $\lambda_{l+1}=0$. We say that $l(\lambda):=l$ is the length of the partition $\lambda$. We also combine repeated entries and use exponents, for instance the partition $(5,5,3,2,1,1,0,0,0,\dots)$ of $17$ will be written $(5^2,3,2,1^2)$. We use the notation $\lambda\vdash n$ to mean $\lambda$ is a partition of $n$. We call $n$ the degree of the partition and write $n=|\lambda|$.

We say that a partition $\lambda=(\lambda_1,\dots,\lambda_l)$ is \emph{$p$-singular} if there exists $t$ such that
	\[\lambda_t=\lambda_{t+1}=\dots=\lambda_{t+p-1}>0\]
i.e. some (non-zero) part of $\lambda$ is repeated $p$ or more times.
Partitions that are not $p$-singular we call \emph{$p$-regular}.

We denote by $\Lambda_n$ the set of all partitions of $n$. We also define $\Lambda_{\leq n}=\cup_{0\leq i\leq n}\Lambda_n$. We will also consider $\Lambda_n^*$ and $\Lambda_{\leq n}^*$,  the subsets of $p$-regular partitions of $\Lambda_n$ and $\Lambda_{\leq n}$  respectively.

There is a  partial order on $\Lambda_{\leq n}$ called the \emph{dominance order (with degree)}, which we denote by $\preceq$. For $\lambda, \mu \in \Lambda_{\leq n}$ we say that $\lambda \prec \mu$ if either $|\lambda|<|\mu|$, or $\lambda \neq \mu$, $|\lambda|=|\mu|$ and $\sum_{i=1}^j \lambda_i \leq \sum_{i=1}^j \mu_j$ for all $j\geq 1$. We write $\lambda \preceq \mu$ to mean $\lambda \prec \mu$ or $\lambda = \mu$.

To each partition $\lambda$ we may associate the Young diagram 
	\[[\lambda]=\{(x,y)~|~x,y\in\mathbb{Z},~1\leq x\leq l,~1\leq y \leq \lambda_x\}.\]
An element $(x,y)$ of $[\lambda]$ is called a \emph{node}. If $\lambda_{i+1}<\lambda_{i}$, then the node $(i,\lambda_i)$ is called a \emph{removable} node of $\lambda$. If $\lambda_{i-1}>\lambda_i$, then we say the node $(i,\lambda_i+1)$ of $[\lambda]\cup\{(i,\lambda_i+1)\}$ is an \emph{addable} node of $\lambda$. This is illustrated in Figure \ref{fig:youngdiag}. If a partition $\mu$ is obtained from $\lambda$ by removing a removable (resp. adding an addable) node then we write $\mu\triangleleft\lambda$ (resp. $\mu\triangleright\lambda$). We will also write $\mu = \lambda - \varepsilon_i$ (resp. $\mu=\lambda + \varepsilon_i$) if $\mu$ is obtained from $\lambda$ by removing (resp. adding) a node in row $i$.

	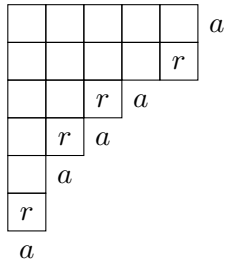
\begin{figure}
		\centering
		\begin{tikzpicture}[scale=0.5]
			\foreach \x in {0,...,4}
				{\draw (\x,5) rectangle +(1,-1);}
			\foreach \x in {0,...,4}
				{\draw (\x,4) rectangle +(1,-1);}
			\foreach \x in {0,...,2}
				{\draw (\x,3) rectangle +(1,-1);}
			\draw (0,2) rectangle +(1,-1);
			\draw (1,2) rectangle +(1,-1);
			\draw (0,1) rectangle +(1,-1);
			\draw (0,0) rectangle +(1,-1);
			
			\draw (4.5,3.45) node {$r$};
			\draw (2.5,2.45) node {$r$};
			\draw (1.5,1.45) node {$r$};
			\draw (0.5,-0.55) node {$r$};

			\draw (5.5,4.45) node {$a$};
			\draw (3.5,2.45) node {$a$};
			\draw (2.5,1.45) node {$a$};
			\draw (1.5,0.45) node {$a$};
			\draw (0.5,-1.55) node {$a$};
		\end{tikzpicture}
		\caption{The Young diagram of $\lambda=(5^2,3,2,1^2)$. Removable nodes are marked by $r$ and addable nodes by $a$.}
		\label{fig:youngdiag}
	\end{figure}

Each node $\varepsilon=(x,y)$ of $[\lambda]$ has an associated integer, $c(\varepsilon)$, called the \emph{content} of $\varepsilon$, given by $c(\varepsilon)=y-x$. We write \begin{equation}\label{eq:sumofcontents}
	\mathrm{ct}(\lambda)=\sum_{\varepsilon\in[\lambda]}c(\varepsilon).
\end{equation}

For partitions $\lambda$ and $\mu$, we write $\mu\subseteq \lambda$ if the Young diagram of $\mu$ is contained in the Young diagram of $\lambda$, i.e. if $\mu_i \leq \lambda_i$ for all $i\geq 1$. We write $\mu\subset \lambda$ if $\mu\subseteq \lambda$ and $\mu\neq \lambda$.

\subsection*{Abacus}\label{sec:abacus}

Following \cite[Section 2.7]{james1981representation} we can associate to each partition and prime number $p$ an abacus diagram, consisting of $p$ columns, known as runners, and a configuration of beads across these. By convention we label the runners from left to right, starting with 0, and the positions on the abacus are also numbered from left to right, working down from the top row, starting with $0$ (see Figure \ref{fig:symabacus}). Given a partition $\lambda=(\lambda_1,\dots,\lambda_l)\vdash n$, fix a positive integer $b\geq n$ and define the  $\beta$-sequence of $\lambda$ to be the $b$-tuple
	\[\beta(\lambda,b)=(\lambda_1-1+b,\lambda_2-2+b,\dots,\lambda_l-l+b,-(l+1)+b,\dots).\]
Then place a bead on the abacus in each position given by $\beta(\lambda,b)$, so that there are a total of $b$ beads across the runners. Note that for a fixed value of $b$, the abacus is uniquely determined by $\lambda$, and any such abacus arrangement corresponds to a partition simply by reversing the above. Here is an example of such a construction.

	\begin{example}
		In this example we will fix the values $p=5,n=9,b=10$ and represent the partition $\lambda=(5,4)$ on the abacus. Following the above process, we first calculate the $\beta$-sequence of $\lambda$:
			\begin{eqnarray*}
				\beta(\lambda,10)	&=&(5-1+10,~4-2+10,\;-3+10,\;-4+10,\dots,\;-9+10,\;-10+10)\\
							&=&(14,12,7,6,5,4,3,2,1,0).
			\end{eqnarray*}
		The next step is to place beads on the abacus in the corresponding positions. We also number the beads, so that bead $1$ occupies position $\lambda_1-1+b$, bead $2$ occupies position $\lambda_2-2+b$ and so on. The labelled spaces and the final abacus with labelled beads are shown in Figure \ref{fig:symabacus}.
			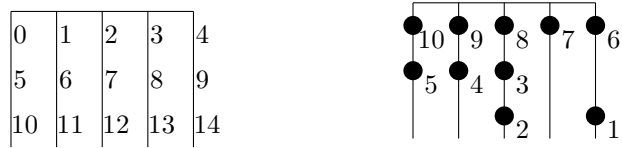
\begin{figure}
				\centering
				\begin{tikzpicture}[scale=0.6]
					\draw (0,3)--(4,3);
					\foreach \x in {0,...,4}
						{\draw (\x,3)--(\x,0);
						\draw (\x+0.2,2.5) node {\small$\x$};}
					\foreach \x in {5,...,9}
						{\draw (\x-5+0.2,1.5) node {\small$\x$};}
					\foreach \x in {10,...,14}
						{\draw (\x-10+0.3,0.5) node {\small$\x$};}
				\end{tikzpicture}
				\hspace{2cm}
				\begin{tikzpicture}[scale=0.6]
					\draw (0,3)--(4,3);
					\foreach \x in {0,...,4}
						{\draw (\x,3)--(\x,0);}
					\foreach \x in {0,...,4}
						{\fill[black] (\x,2.5) circle (6pt);}
					\foreach \x in {0,1,2}
						{\fill[black] (\x,1.5) circle (6pt);}
					\foreach \x in {2,4}
						{\fill[black] (\x,0.5) circle (6pt);}
						
					\foreach \x in {10,...,6}
						{\draw (10-\x+0.4,2.5-0.3) node {\small$\x$};}
					\foreach \x in {5,4,3}
						{\draw (5-\x+0.4,1.5-0.3) node {\small$\x$};}
					\draw (2.4,0.2) node {\small$2$};
					\draw (4.4,0.2) node {\small$1$};
				\end{tikzpicture}

				\caption{The positions on the abacus with 5 runners and the arrangement of beads (numbered) representing $\lambda=(5,4)$.}\label{fig:symabacus}
			\end{figure}
	\end{example}
After fixing values of $p$ and $b$, we will abuse notation and write $\lambda$ for both the partition and the corresponding abacus with $p$ runners and $b$ beads. We then define \newline $\Gamma(\lambda,b)=(\Gamma(\lambda,b)_0,\Gamma(\lambda,b)_1,\dots,\Gamma(\lambda,b)_{p-1})$, where 
	\begin{equation}\label{eq:gammaabacus}
		\Gamma(\lambda,b)_i=\big|\{j:\beta(\lambda,b)_j\equiv i\text{ mod }p\}\big|
	\end{equation}
so that $\Gamma(\lambda,b)$ records the number of beads on each runner of the abacus of $\lambda$.

We define the  \emph{$p$-core} of the partition $\lambda$ to be the partition $\mu$ whose abacus is obtained from that of $\lambda$ by sliding all the beads as far up their runners as possible. In particular, we have $\Gamma(\lambda, b)=\Gamma(\mu,b)$. It can be shown that the $p$-core $\mu$ is independent of the choice of $b$, and so depends only on $\lambda$ and $p$. The $5$-core of the partition $(5,4)$ given in the example above is the partition $(3,1)$. Their abaci are illustrated in Figure \ref{fig:symcore}. 

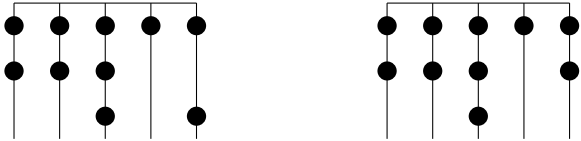
\begin{figure}
\centering

\begin{tikzpicture}[scale=0.6]
					\draw (0,3)--(4,3);
					\foreach \x in {0,...,4}
						{\draw (\x,3)--(\x,0);}
					\foreach \x in {0,...,4}
						{\fill[black] (\x,2.5) circle (6pt);}
					\foreach \x in {0,1,2}
						{\fill[black] (\x,1.5) circle (6pt);}
					\foreach \x in {2,4}
						{\fill[black] (\x,0.5) circle (6pt);}

				\end{tikzpicture}
\hspace{2cm}
\begin{tikzpicture}[scale=0.6]
					\draw (0,3)--(4,3);
					\foreach \x in {0,...,4}
						{\draw (\x,3)--(\x,0);}
					\foreach \x in {0,...,4}
						{\fill[black] (\x,2.5) circle (6pt);}
					\foreach \x in {0,1,2,4}
						{\fill[black] (\x,1.5) circle (6pt);}
					\foreach \x in {2}
						{\fill[black] (\x,0.5) circle (6pt);}
				\end{tikzpicture}
\caption{The abaci of the partition $\lambda=(5,4)$ and its $5$-core, the partition $(3,1)$.}\label{fig:symcore}
\end{figure}

\section{Representation theory of the symmetric group}

We denote by $\sym_n$ the symmetric group of degree $n$. In this section, we will briefly recall some results in the representation theory of $\sym_n$ which will be needed later in the paper.

The group algebra $R\sym_n$ is a cellular algebra, as shown in \cite{Graham1996cellular}. The cell modules are indexed by partitions $\lambda\in \Lambda_n$ and are more commonly known as \emph{Specht modules}. We denote the Specht module indexed by $\lambda$ by $S^\lambda_R$. These can be constructed explicitly, see for example \cite[Chapter 4]{james1978representation}. We define the $K\sym_n$-module $S_{K}^\lambda := K\otimes_{R}S_R^\lambda$ and the $k\sym_n$-module $S_k^\lambda := k\otimes_R S_R^\lambda$.  These are the cell modules for $K\sym_n$ and $k\sym_n$ respectively.

\begin{thm}[{\cite[Theorem 4.12]{james1978representation}}]
	The set of all $S_K^\lambda$, $\lambda \in \Lambda_n$, gives a complete set of pairwise non-isomorphic simple $K\sym_n$-modules.
\end{thm}

\begin{thm}[{\cite[Theorem 11.5]{james1978representation}}] 
	For $\lambda\in \Lambda_n^*$, the Specht module $S_k^\lambda$ has simple head, denoted by $D_k^\lambda$. Moreover the set of all $D_k^\lambda$, $\lambda\in \Lambda_n^*$ gives a complete set of pairwise non-isomorphic simple $k\sym_n$-modules.
\end{thm}

The problem of describing the decomposition numbers for $k\sym_n$ remains wide open. But the (cell-)blocks of this algebra are well-known.

	\begin{thm}[Nakayama's Conjecture]\cite[Chapter 6]{james1981representation}\label{nakayama}
		Let $\lambda, \mu\in \Lambda_n$. Then $\mu\in \mathcal{B}_\lambda(k\sym_n)$ if and only if $\lambda$ and $\mu$ have the same $p$-core, that is $\Gamma(\lambda, b)=\Gamma(\mu,b)$ for some (and hence all) $b\geq n$.
		\label{thm:nakayama}
	\end{thm}
		
We now give another characterisation of the cell-blocks of $k\sym_n$, which will be useful when considering the block of the partition algebra. Write any partition $\lambda=(\lambda_1, \lambda_2, \ldots , \lambda_l)\in \Lambda_n$ as an $n$-tuple  by defining $\lambda_i=0$ for all $l<i\leq n$. Now define the $n$-tuple $\rho_n = (-1,-2,-3, \ldots , -n)$. For $x,y\in \mathbb{Z}_n$ we write  $x\sim_p y$ if and only if there exists a permutation $\sigma\in \sym_n$ such that $x_i\equiv y_{\sigma(i)} \, \mbox{mod} \, p$ for all $1\leq i\leq n$. Then Theorem \ref{nakayama} can be rephrased as follows.

\begin{thm}\label{newnakayama}
For any $\lambda,\mu\in \Lambda_n$ we have that $\mu\in \mathcal{B}_\lambda(k\sym_n)$ if and only if $\lambda + \rho_n \sim_p \mu+\rho_n$
\end{thm}

	\subsection*{Reflection geometry.}

We now give a final characterisation of the cell-blocks of $k\sym_n$ in terms of the action of an affine reflection group.
Let $\{\varepsilon_1, \varepsilon_2, \ldots , \varepsilon_n\}$ be a set of formal symbols and set
$$E_n=\bigoplus_{i=1}^{n} \mathbb{R}\varepsilon_i$$
to be the $n$-dimensional real vector space with basis $\varepsilon_1, \varepsilon_2, \ldots ,\varepsilon_n$. 
We have an inner product $\langle \, , \, \rangle$ given by extending linearly the relations
$$\langle \varepsilon_i, \varepsilon_j\rangle = \delta_{ij}$$
for all $1\leq i, j \leq n$, where $\delta_{ij}$ is the Kronecker delta. Let $\Phi_n = \{\varepsilon_i - \varepsilon_j\, : 1\leq i,j\leq n\}$ be a root system of type $A_{n-1}$, and define $W_n$ to be the corresponding Weyl group generated by the reflections $s_{i,j}$ (for $1\leq i,j\leq n$) given by
$$s_{i,j}(x)=x-\langle x,\varepsilon_i-\varepsilon_j\rangle (\varepsilon_i - \varepsilon_j)$$
for all $x\in E_n$. The affine Weyl group $W_n^p$ is generated by $W_n$ and translations by $p(\varepsilon_i - \varepsilon_j)$ (for all $1\leq i,j \leq n$). It is also generated by the affine reflections
$s_{i,j,rp}$ (for $1\leq i,j\leq n$ and $r\in \mathbb{Z}$) given by
$$s_{i,j,rp}(x)=x-(\langle x,\varepsilon_i-\varepsilon_j\rangle -rp)(\varepsilon_i - \varepsilon_j)$$
for all $x\in E_n$. 

 We will identify each element $x=x_1\varepsilon_1 + x_2 \varepsilon_2 + \ldots + x_n\varepsilon_n\in E_n$ with the $n$-tuple $(x_1, x_2, \ldots , x_n)$. Now  consider the shifted action of $W_n^p$ by $\rho_n=(-1,-2,\ldots , -n)$ given by
$$w \cdot x = w(x+\rho_n)-\rho_n$$
for all $w\in W_n^p$ and $x\in E_n$. We view each partition $\lambda\in \Lambda_n$ as an $n$-tuple and hence as an element of $E_n$. With these definitions we can now reformulate Theorem \ref{newnakayama} as follows.

\begin{thm}\label{geomnakayama}
Let $\lambda ,\mu\in \Lambda_n$. Then we have $\mu\in \mathcal{B}_\lambda(k\sym_n)$ if and only if $\mu\in W_n^p \cdot \lambda$.
\end{thm}

\begin{proof}
We need to prove that $\mu +\rho_n \sim_p \lambda + \rho_n$ if and only if $\mu \in W_n^p \cdot \lambda$.
It is clear that if $\mu\in W_n^p \cdot \lambda$ then $\mu + \rho_n \sim_p \lambda + \rho_n$.
Now if $\mu + \rho_n \sim_p \lambda + \rho_n$ then, by definition, there exists $w\in W_n$ and $x\in \mathbb{Z}^n$ with $\mu +\rho_n = w(\lambda + \rho_n) +px$. Now as $\sum_{i=1}^n\lambda_i = \sum_{i=1}^n\mu_i$, we must have $\sum_{i=1}^n x_i=0$, and so $x\in \mathbb{Z}\Phi_n$ as required. 
\end{proof}

\section{The partition algebra: definition and cellularity}

In this section we give the definitions of the partition algebra and the half partition algebras and recall their cellular structure.

\subsection*{Definitions and first properties}

For a fixed $n\in\mathbb{N}$ and $\delta\in \mathbb{F}$, we define the partition algebra $P_n^\mathbb{F}(\delta)$ to be the set of $\bbf$-linear combinations of set-partitions of $\{1,2,\dots,n,\bar1,\bar2,\dots,\bar n\}$. We call each connected component of a set-partition a \emph{block}. For instance,
	\[\big\{\{1,3,\bar3,\bar4\},\{2,\bar1\},\{4\},\{5,\bar2,\bar5\}\big\}\]
is a set-partition with $n=5$ consisting of 4 blocks. Any block with $\{i,\bar j\}$ as a subset for some $i$ and $j$ is called a \emph{propogating block}. We can represent each set-partition by a \emph{partition diagram} (or $n$-partition diagram), consisting of two rows of $n$ nodes, $n$ northern nodes indexed by $1,2, \ldots ,n$ and $n$ southern nodes indexed by $\bar{1}, \bar{2}, \ldots ,\bar{n}$, with arcs between nodes in the same block. Note that in general there are many partition diagrams corresponding to the same set-partition. For example if we take $n=5$, then some of the diagrams representing the set-partition   $\big\{\{1,3,\bar3,\bar4\},\{2,\bar1\},\{4\},\{5,\bar2,\bar5\}\big\}$  are given in Figure \ref{fig:threediagrams}. We will identify partition diagrams corresponding to the same set-partition.
	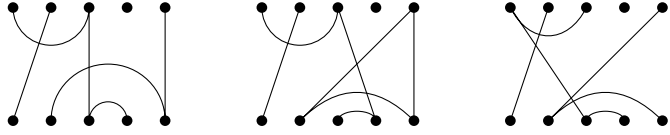
\begin{figure}
		\centering
		\begin{tikzpicture}[scale=0.5]
			\foreach \x in {1,...,5}
				{\fill[black] (\x,3) circle (4pt);
				\fill[black] (\x,0) circle (4pt);}
				\draw (1,3) arc (-180:0:1) -- (3,0) arc (180:0:0.5);
				\draw (2,3)--(1,0);
				\draw (2,0) arc (180:0:1.5) -- (5,3);
		\end{tikzpicture}\hspace{1cm}
		\begin{tikzpicture}[scale=0.5]
			\foreach \x in {1,...,5}
				{\fill[black] (\x,3) circle (4pt);
				\fill[black] (\x,0) circle (4pt);}
				\draw (1,3) arc (-180:0:1) -- (4,0) arc (0:180:0.5 and 0.25);
				\draw (2,3)--(1,0);
				\draw (2,0) .. controls (3,1) and (4,1) .. (5,0) -- (5,3) -- (2,0);
		\end{tikzpicture}\hspace{1cm}
		\begin{tikzpicture}[scale=0.5]
			\foreach \x in {1,...,5}
				{\fill[black] (\x,3) circle (4pt);
				\fill[black] (\x,0) circle (4pt);}
				\draw (3,3) .. controls (2.5,2) and (1.5,2) .. (1,3) -- (3,0) arc (180:0:0.5 and 0.25);
				\draw (2,3)--(1,0);
				\draw (5,0) .. controls (4,1) and (3,1) .. (2,0) -- (5,3);
		\end{tikzpicture}
		\caption{Three diagrams  representing the  set-partition $\big\{\{1,3,\bar3,\bar4\},\{2,\bar1\},\{4\},\{5,\bar2,\bar5\}\big\}$}\label{fig:threediagrams}
	\end{figure}

 Multiplication in the partition algebra is given by concatenation of diagrams in the following way: to obtain the result $x\cdot y$ given diagrams $x$ and $y$, place $x$ on top of $y$ and identify the southern  nodes of $x$ with the northern nodes of $y$. This new diagram may contain a number, $t$ say, of blocks in the centre not connected to the northern or southern edges of the diagram. These we remove and multiply the final result by $\delta^t$. An example is given in Figure \ref{fig:partmult}.

\begin{figure}
	\centering
	\begin{tikzpicture}[scale=0.5]
	 \foreach \x in {1,...,5,7,8,...,11,14,15,...,18}
	 	{\fill[black] (\x,0) circle (4pt);
	 	\fill[black] (\x,3) circle (4pt);}
	 	
	\draw (3,3) arc (0:-180:0.5 and 0.25) -- (3,0);\draw (4,3)--(1,0);\draw (5,3)--(5,0);
	\draw (6,1.5) node {$\times$};
	\draw (9,3) .. controls (8.5,2) and (7.5,2) .. (7,3) -- (9,0) arc (180:0:0.5 and 0.25);
	\draw (8,3)--(7,0);
	\draw (11,0) .. controls (10,1) and (9,1) .. (8,0) -- (11,3);
	\draw (12,1.5) node {$=$};\draw (13,1.5) node {$\delta$};
	\draw (18,0) .. controls (17,1) and (16,1) .. (15,0) -- (18,3);
	\draw (17,3) arc (0:-180:0.5 and 0.25) arc (0:-180:0.5 and 0.25) -- (16,0) arc (180:0:0.5 and 0.25);
	\end{tikzpicture}
	\caption{Multiplication of two diagrams in $P_5^R(\delta)$.}\label{fig:partmult}
\end{figure}
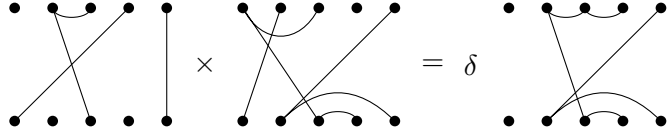

It is easy to see that the elements $s_{i,j}$, $p_{i,j}$ ($1\leq i<j\leq n$) and $p_i$ ($1\leq i\leq n$) defined in Figure \ref{fig:generators} generate $P_n^\bbf(\delta)$ (see \cite[Theorem 1.11]{halverson2005partition}).
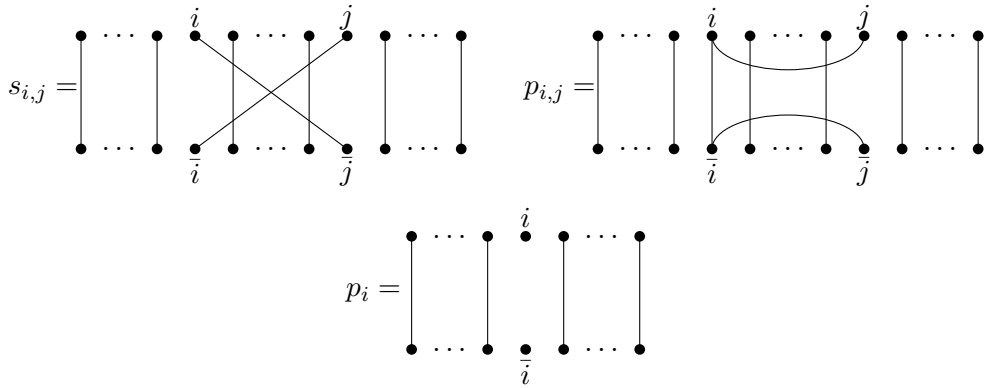
\begin{figure}
\centering
\begin{tikzpicture}[scale=0.5]
	\draw (-1,1.5) node {$s_{i,j}=$};
	\foreach \x in {0,2,3,4, 5.5+.5,6.5+.5,8,10}
	{\fill[black] (\x,0) circle (4pt);
	\fill[black] (\x,3) circle (4pt);}
	\foreach \x in {0,2,4, 5.5+.5,7.5+.5,10}
	{\draw (\x,3)--(\x,0);}
	\draw (3,3)--(7,0);\draw (7,3)--(3,0);
	\draw (3,3.5) node {$\tiny i$};\draw (7,3.5) node {$j$};
	\draw (3,-0.6) node {$\bar i$};\draw (7,-0.6) node {$\bar j$};
\draw (1,3) node {$\ldots$};\draw (1,0) node {$\ldots$};
\draw (5,3) node {$\ldots$};\draw (5,0) node {$\ldots$};
\draw (9,3) node {$\ldots$};\draw (9,0) node {$\ldots$};
\end{tikzpicture}\hspace{0.5cm}
\begin{tikzpicture}[scale=0.5]
	\draw (-1,1.5) node {$p_{i,j}=$};
	\foreach \x in {0,2,3,4, 5.5+.5,6.5+.5,8,10}
	{\fill[black] (\x,0) circle (4pt);
	\fill[black] (\x,3) circle (4pt);}
	\foreach \x in {0,2,4, 5.5+.5,7.5+.5,10}
	{\draw (\x,3)--(\x,0);}
	\draw (3,3)--(3,0);   \draw (3,3) arc (0:180:-2 and -0.9) ;  \draw (3,0) arc (0:180:-2 and 0.9) ;
	\draw (3,3.5) node {$\tiny i$};\draw (7,3.5) node {$j$};
	\draw (3,-0.6) node {$\bar i$};\draw (7,-0.6) node {$\bar j$};
\draw (1,3) node {$\ldots$};\draw (1,0) node {$\ldots$};
\draw (5,3) node {$\ldots$};\draw (5,0) node {$\ldots$};
\draw (9,3) node {$\ldots$};\draw (9,0) node {$\ldots$};
\end{tikzpicture}
 \hspace{0.5cm}
\begin{tikzpicture}[scale=0.5]
	\draw (-1,1.5) node {$p_{i}=$};
	\foreach \x in {0,2,3,4,6}
	{\fill[black] (\x,0) circle (4pt);
	\fill[black] (\x,3) circle (4pt);}
	\foreach \x in {0,2,4,6}
	{\draw (\x,3)--(\x,0);}
	\draw (3,3.5) node {$i$};
	\draw (3,-0.6) node {$\bar i$};
\draw (1,0) node {$\ldots$};\draw (1,3) node {$\ldots$};
\draw (5,0) node {$\ldots$};\draw (5,3) node {$\ldots$};
\end{tikzpicture}
\caption{Generators of the partition algebra}\label{fig:generators}
\end{figure}
We write $s_i:=s_{i,i+1}$ and $p_{i+\half}:=p_{i,n}$ for $1\leq i\leq n-1$.

Notice that multiplication in $P_n^\bbf(\delta)$ cannot increase the number of propagating blocks. We therefore have a filtration of $P_n^\bbf(\delta)$ by the number of propagating blocks.

In what follows we will assume that $\delta\in \bbf$ is non-zero. This allows us to realise the filtration by use of the idempotents $e_t=\frac{1}{\delta^t}p_1p_2\ldots p_t$ (for $1\leq t\leq n$). 

So we have
\begin{equation}\label{eq:partfiltration}
		J_n^{(0)}\subset J_n^{(1)}\subset \dots \subset J_n^{(n-1)}\subset J_n^{(n)}=P_n^\bbf(\delta)
	\end{equation}
where $J_n^{(t)}=P_n^\bbf(\delta)e_{n-t}P_n^\bbf(\delta)$ is spanned by all diagrams with at most $t$ propogating blocks.

	 We also use $e_t$ to construct algebra isomorphisms
	\begin{equation}\label{eq:partiso}
		\Phi_{n,t}:P_{n-t}^\bbf(\delta)\longrightarrow e_tP_n^\bbf(\delta)e_t
	\end{equation}
	taking a diagram in $P_{n-t}^\bbf(\delta)$ and adding $t$ extra northern and southern nodes to the lefthand end.
Using these isomorphisms, and following \cite[Section 6.2]{green1980polynomial}, we obtain exact localisation functors
	\begin{eqnarray}
		F_{n,n-t}:P_n^\bbf(\delta)\text{-\catmod}	&\longrightarrow&	P_{n-t}^\bbf(\delta)\text{-\catmod}\label{eq:partfn}\\
					 			 M 	&\longmapsto&    	e_tM, \nonumber
	\end{eqnarray}
and  right exact globalisation functors
	\begin{eqnarray}
		G_{n-t,n}:P_{n-t}^\bbf(\delta)\text{-\catmod}	&\longrightarrow&	P_{n}^\bbf(\delta)\text{-\catmod}\label{eq:partgn}\\
					  M 				&\longmapsto&	P_{n}^\bbf(\delta)e_{t}\otimes_{P_{n-t}^\bbf(\delta)}M. \nonumber
	\end{eqnarray}
Since $F_{n,n-t}G_{n-t,n}(M)\cong M$ for all $M\in P_{n-t}^\bbf(\delta)$-\catmod, $G_{n-t}$ is a full embedding of categories. From the filtration \eqref{eq:partfiltration} we see that
	\begin{equation}\label{eq:partlift}
		P_n^\bbf(\delta)/J_n^{(n-1)}\cong\bbf \sym_n.
	\end{equation}
Thus using \eqref{eq:partiso} with $t=1$, we obtain by induction that the simple $P_n^\bbf(\delta)$-modules are indexed by the set $\Lambda_{\leq n}$ if $\mathbb{F}=K$ and by the set $\Lambda_{\leq n}^*$ if $\mathbb{F}=k$.

\subsection*{Cellular structure}

It was shown in \cite{xi1999partition} that the partition algebra $P_n^\bbf(\delta)$ is cellular. The cell modules $\Delta_\lambda^\bbf(n;\delta)$ are  indexed by the set of  partitions $\lambda\in\Lambda_{\leq n}$, and the partial order is given by the dominance order (with degree)  $\preceq$ defined in Section 2. When $\lambda\vdash n$, we obtain $\Delta_\lambda^\bbf(n;\delta)$ by lifting the Specht module $S^\lambda_\bbf$ to the partition algebra using \eqref{eq:partlift}. When $\lambda\vdash n-t$ for some $t>0$, we obtain the cell module by
\begin{equation}\label{partcell}
\Delta_\lambda^\bbf(n;\delta)= G_{n-t,n}(S_\bbf^\lambda) = P_n^\bbf(\delta) e_t \otimes_{P_{n-t}^\bbf (\delta)} S^\lambda_\bbf.
\end{equation}
Over $K$, each of the cell modules has a simple head $L_\lambda^K(n;\delta)$, and these form a complete set of non-isomorphic simple $P_n^K(\delta)$-modules. Over $k$, the heads $L_\lambda^k(n;\delta)$ of cell modules labelled by $p$-regular partitions $\lambda\in\Lambda_{\leq n}^\ast$ provide a complete set of non-isomorphic simple $P_n^k(\delta)$-modules. 

\noindent From the cellular structure we have that $[\Delta_\lambda^\bbf(n,\delta): L_\mu^\bbf(n,\delta)]\neq 0$ implies $\mu \succeq \lambda$.

	When the context is clear, we will write $\Delta_\lambda^\bbf(n)$ and $L_\lambda^\bbf(n)$ to mean $\Delta_\lambda^\bbf(n;\delta)$ and $L_\lambda^\bbf(n;\delta)$ respectively. By definition we have that the localisation and globalisation functors preserve the cell modules. More precisely, for $\lambda\in \Lambda_{\leq n}$, we have
\begin{align*}
	F_{n,n-t}(\Delta_\lambda^\bbf(n))&\cong\begin{cases}
									\Delta_\lambda^\bbf(n-t)&\text{ if }\lambda\in\Lambda_{\leq n-t}\\
									0&\text{ otherwise,}
									\end{cases}\\
	G_{n,n+t}(\Delta_\lambda^\bbf(n))&\cong\Delta_\lambda^\bbf(n+t).
\end{align*}		

We also have an explicit construction of the cell modules, which follows directly from the definition given in (\ref{partcell}). Let $I(n,n-t)$ be the set of partition diagrams with precisely $n-t$ propagating blocks and $\overline{1},\overline{2},\dots,\overline{t}$ each in singleton blocks. Then denote by $V^\bbf(n,n-t)$ the $\bbf$-space with basis $I(n,n-t)$. 
For a partition $\lambda\vdash n-t$ we can easily show that $\Delta_\lambda^\bbf(n)\cong V^\bbf(n,n-t)\otimes_{\sym_{n-t}}S_\bbf^\lambda$, where $S_\bbf^\lambda$ is the Specht module and the right action of $\sym_{n-t}$ on $V^\bbf(n,n-t)$ is by permutation of the $n-t$ rightmost southern nodes. The action of $P_n^\bbf(\delta)$ on $\Delta_\lambda^\bbf(n)$ is as follows: given a partition diagram $x\in P_n^\bbf(\delta)$, $v\in I(n,n-t)$ and $s\in S_\bbf^\lambda$, we define the element
	\[x(v\otimes s)=(xv)\otimes s\]
where $(xv)$ is the product of the partition diagrams if the concatenation of $x$ and $v$ has $n-t$ propagating blocks, and is 0 otherwise. (Note that when $\lambda \vdash n$ we have $\Delta_\lambda^\bbf(n)=S_\lambda^\bbf$).

Now if we take $\delta\in R$ and $V^R(n,n-t)$ to be the free $R$-module with basis $I(n,n-t)$ then we can define an $R$-form for the cell modules $\Delta_\lambda^R(n;\delta):=V^R(n,n-t)\otimes_{\sym_{n-t}}S^\lambda_R$, and we have
\[\Delta_\lambda^K(n;\delta)=K\otimes_R\Delta_\lambda^R(n;\delta)~~~~~ \text{and}~~~~~\Delta_\lambda^k(n;\delta)=k\otimes_R\Delta_\lambda^R(n;\delta),\]
where we denote by $\delta$ both its embedding in $K$ and its projection onto $k$.

\medskip

\subsection*{The half partition algebras}

We will also need to consider the algebra $P_{n-\frac{1}{2}}^\bbf(\delta)$, which is the subalgebra of $P_n^\bbf(\delta)$ spanned by all set-partitions with $n$ and $\bar n$ in the same block. This algebra was introduced by P. Martin in \cite{martin2000partition}. The half partition algebra also has a filtration by the number of propagating blocks, and when $\delta\in\bbf$ is non-zero, this filtration can be realised as
	\begin{equation}\label{eq:partfiltrationhalf}
		J_{n-\frac{1}{2}}^{(1)}\subset J_{n-\frac{1}{2}}^{(2)}\subset \dots J_{n-\frac{1}{2}}^{(n-1)}\subset J_{n-\frac{1}{2}}^{(n)}=P_{n-\frac{1}{2}}^\bbf(\delta)
	\end{equation}
where $J_{n-\frac{1}{2}}^{(t)}=P_{n-\half}^\bbf(\delta) e_{n-t}P_{n-\half}^\bbf(\delta)$. Note that since we require the nodes $n$ and $\bar n$ to be in the same block, we always have at least one propagating block. 
For the same reason, we see that 
\begin{equation}\label{quotienthalf}
P_{n-\frac{1}{2}}^\bbf(\delta)/J_{n-\frac{1}{2}}^{(n-1)}\cong\bbf\sym_{n-1}.
\end{equation}
Thus we have that the simple $P_{n-\frac{1}{2}}^\bbf(\delta)$-modules are indexed by $\Lambda_{\leq n-1}$ if $\bbf = K$ and by  $\Lambda_{\leq n-1}^\ast$ if $\bbf=k$.

It is also easy to see that $e_tP_{n-\half}^\bbf(\delta)e_t\cong P_{n-t-\half}^\bbf(\delta)$, and so we can define localisation functors $F_{n-\half, n-t-\half}$ and globalisation functors $G_{n-t-\half, n-\half}$ as for the partition algebras.

The algebra $P_{n-\frac{1}{2}}^\bbf(\delta)$ is also cellular \cite{martin2000partition}. We can construct the cell modules in a similar way. Let $I(n-\frac{1}{2},n-t)$ be the set of partition diagrams with precisely $n-t$ propagating blocks, one of which containing $n$ and $\bar n$, and with $\overline{1},\overline{2},\dots,\overline{t}$ each in singleton blocks. Then denote by $V^\bbf(n-\frac{1}{2},n-t)$ the $\mathbb{F}$-module with basis $I(n-\frac{1}{2},n-t)$.
For a partition $\lambda\vdash n-t-1$ we can define $\Delta_\lambda^\bbf(n-\frac{1}{2};\delta)\cong V^\bbf(n-\frac{1}{2},n-t)\otimes_{\sym_{n-t-1}}S_\bbf^\lambda$, where $S_\bbf^\lambda$ is a Specht module and the right action of $\sym_{n-t-1}$ on $V(n-\half,n-t)$ is by permuting the $n-t-1$ southern nodes $\overline{t+1}, \ldots, \overline{n-1}$. The left  action of $P_{n-\frac{1}{2}}^\bbf(\delta)$ is the same as in the partition algebra case.

\subsection*{Induction and restriction}
By definition we have an inclusion $P_{n-\half}^\mathbb{F}(\delta)\subset P_{n}^\mathbb{F}(\delta)$. Note that 
we also have an injective algebra homomorphism 	
\begin{equation}\label{incl}	
		\iota \, : \, P_n^\bbf(\delta)\longrightarrow P_{n+\frac{1}{2}}^\bbf(\delta)\, : \,
		d\longmapsto d\cup \big\{\{n+1,\overline{n+1}\}\big\}.
\end{equation}	
This allows us to define the following restriction and induction functors.
\begin{eqnarray*}
&& \mathrm{res}_n  : P_n^\bbf(\delta)\text{-\catmod}\longrightarrow P^\bbf_{n-\frac{1}{2}}(\delta)\text{-\catmod}\, : \,
				M\longmapsto M|_{P_{n-\half}^\bbf(\delta)}.\\
&& \mathrm{ind}_{n-\half}:P^\bbf_{n-\half}(\delta)\text{-\catmod} \longrightarrow P^\bbf_{n}(\delta)\text{-\catmod}\, : \, M \longmapsto P_{n}^\bbf(\delta)\otimes_{P_{n-\tiny{\half}}^\bbf(\delta)}M.\\
&& \mathrm{res}_{n+\half}  : P_{n+\half}^\bbf(\delta)\text{-\catmod}\longrightarrow P^\bbf_{n}(\delta)\text{-\catmod}\, : \,
				M\longmapsto M|_{P_{n}^\bbf(\delta)}.\\
&& \mathrm{ind}_{n}:P^\bbf_{n}(\delta)\text{-\catmod} \longrightarrow P^\bbf_{n+\half}(\delta)\text{-\catmod}\, : \, M \longmapsto P_{n+\half}^\bbf(\delta)\otimes_{P_{n}^\bbf(\delta)}M.
\end{eqnarray*}

\noindent In \cite{martin2000partition} P. Martin gives branching rules for the cell modules under the above restriction functors (see also \cite[Proposition 3.4]{enyang2013seminormal}). We now recall this result. We will use the notation $M\cong \biguplus_{0\leq i \leq s} N_i$ to denote a module $M$ with a filtration $0=M_{s+1}\subset M_{s}\subset \ldots \subset M_1 \subset M_0=M$ such that $M_i/M_{i+1}\cong N_i$ for all $0\leq i\leq s$.

\begin{thm} \cite[Proposition 7]{martin2000partition}
Let $\lambda\in \Lambda_{\leq n}$. If $|\lambda|=n$ then we have $\mathrm{res}_{n}\Delta_\lambda^\bbf(n) \cong \biguplus_{\mu\triangleleft\lambda}\Delta_\mu^\bbf(n-\textstyle{\half})$. If $|\lambda|\leq n-1$ then we have an exact sequence
\begin{equation}\label{eq:resn}
0\longrightarrow\biguplus_{\mu\triangleleft\lambda}\Delta_\mu^\bbf(n-\textstyle{\half})\longrightarrow \mathrm{res}_{n}\Delta_\lambda^\bbf(n) \longrightarrow\Delta_\lambda^\bbf(n-\half) \longrightarrow0.
\end{equation}
Let $\lambda\in \Lambda_{\leq n-1}$ then we have an exact sequence
\begin{equation}\label{eq:indhalf}
0\longrightarrow\Delta_\lambda^\bbf(n)\longrightarrow \mathrm{ind}_{n-{\half}}\Delta_\lambda^\bbf(n-\textstyle\half)\longrightarrow\displaystyle\biguplus_{\mu\triangleright\lambda}\Delta_\mu^\bbf(n) \longrightarrow0.
\end{equation}
Let $\lambda\in \Lambda_{\leq n}$. If $|\lambda|=n$ then we have $\mathrm{res}_{n+{\half}}\Delta_\lambda^\bbf(n+\textstyle{\half}) \cong \Delta_\lambda^\bbf(n)$. If $|\lambda|\leq n-1$ then we have an exact sequence
\begin{equation}\label{eq:reshalf}
0\longrightarrow\Delta_\lambda^\bbf(n)\longrightarrow \mathrm{res}_{n+{\half}}\Delta_\lambda^\bbf(n+\textstyle{\half}) \longrightarrow\displaystyle\biguplus_{\mu\triangleright\lambda}\Delta_\mu^\bbf(n) \longrightarrow0.
\end{equation}
Let $\lambda\in \Lambda_{\leq n}$ then we have an exact sequence
\begin{equation}\label{eq:indn}
0\longrightarrow\biguplus_{\mu\triangleleft\lambda}\Delta_\mu^\bbf(n+\textstyle{\half})\longrightarrow \mathrm{ind}_n\Delta_\lambda^\bbf(n)\longrightarrow\Delta_\lambda^\bbf(n+\half) \longrightarrow 0.
\end{equation}

Moreover, in the exact sequences (\ref{eq:resn})-(\ref{eq:indn}) above, the filtrations by cell modules  can be chosen so that the cell modules appear in dominance order, with the most dominant factor appearing at the top of the filtration.

\end{thm}

\begin{proof}
The proofs for the branching rules for the restriction functors (\ref{eq:resn}) and (\ref{eq:reshalf}) can be found for example in \cite[Proposition 3.4]{enyang2013seminormal}. The branching rule for the induction functor $\ind_{n-\half}$ given in (\ref{eq:indhalf}) follows directly from the fact that as functors we have $\mathrm{ind}_{n-\half} \cong \mathrm{res}_{n+\half}G_{n-\half ,n+\half}$.  To see this, simply observe that as $(P_{n}^\bbf(\delta), P_{n-\half}^\bbf(\delta))$-bimodules we have  $P_{n+\half}^\bbf(\delta) e_1 \cong P_{n}^\bbf(\delta)$. Similarly, the branching rule for the induction functor $\ind_{n}$ given in (\ref{eq:indn}) follows directly from the fact that as functors we have $\mathrm{ind}_n \cong \mathrm{res}_{n+1}G_{n,n+1}$. To see this, simply observe that as $(P_{n+\half}^\bbf(\delta), P_n^\bbf(\delta))$-bimodules we have $P_{n+1}^\bbf(\delta) e_1 \cong P_{n+\half}^\bbf(\delta)$.
\end{proof}

\subsection*{A Morita equivalence}

 P. Martin proved the existence of a  Morita equivalence between $P_{n+\half}^\bbf(\delta)$ and $P_n^\bbf(\delta -1)$ when $\bbf = \mathbb{C}$. In fact, this equivalence holds over any field $\bbf$.
\begin{thm}\cite[Section 3]{martin2000partition}\label{prop:morita}
	Define the idempotent
		\[\xi_{n+1}=\prod_{i=1}^n(1-p_{i,n+1})\in P_{n+1}^\bbf(\delta).\]
(i) We have an algebra isomorphism
		\[\xi_{n+1}P_{n+\frac{1}{2}}^\bbf(\delta)\xi_{n+1}\cong P_n^\bbf(\delta-1).\]
(ii) The isomorphism given in (i) induces a Morita equivalence between the categories $P_{n+\frac{1}{2}}^\bbf(\delta)\text{-\catmod}$ and $P_n^\bbf(\delta-1)\text{-\catmod}$. More precisely, using the isomorphism given in (i), the functors
	\begin{align*}
		\Phi:P_{n+\frac{1}{2}}^\bbf(\delta)\text{-\catmod}&\longrightarrow P_n^\bbf(\delta-1)\text{-\catmod}\\
												M&\longmapsto\xi_{n+1}P_{n+\frac{1}{2}}^\mathbb{F}(\delta)\otimes_{P_{n+\frac{1}{2}}^\mathbb{F}(\delta)} M\\
		\text{and~~~~~}\Psi:P_n^\bbf(\delta-1)\text{-\catmod}&\longrightarrow P_{n+\frac{1}{2}}^\bbf(\delta)\text{-\catmod}\\
												N&\longmapsto P_{n+\textstyle\half}^\bbf(\delta)\xi_{n+1}\otimes_{P_n^\bbf(\delta-1)}N
	\end{align*}
	define an equivalence of categories. 

\noindent (iii) Cell modules are preserved under the equivalence given in (ii). More precisely, we have
		\[\Phi(\Delta_\lambda^\bbf(n+\textstyle\half,\delta))\cong\Delta_\lambda^\bbf(n,\delta -1)\]
	for all $\lambda\in\Lambda_{\leq n}$.
\end{thm}

\begin{proof}
The proofs of (i) and (ii) are given in \cite[Section 3]{martin2000partition}. Although Martin works over the field of complex numbers throughout his paper, his proof of this result is in fact characteristic free. We briefly recall his arguments here. First note that as vector spaces we have $P_n^\mathbb{F}(\delta -1)=P_n^\mathbb{F}(\delta)$ so we can consider the inclusion of vector spaces $\iota : P_n^\mathbb{F}(\delta -1) \rightarrow P_{n+\frac{1}{2}}^\mathbb{F}(\delta)$ as defined in (\ref{incl}). Now define the map
\begin{equation}\label{algebraiso}
\theta \, : \, P_n^\mathbb{F}(\delta -1) \rightarrow \xi_{n+1}P_{n+\frac{1}{2}}^\mathbb{F}(\delta)\xi_{n+1}\, : \, x\mapsto \xi_{n+1} \iota(x) \xi_{n+1}.
\end{equation}
Note that  for any $(n+1)$-partition diagram we have $\xi_{n+1}x = 0$ (resp. $x\xi_{n+1}=0$) if $x$ has a block containing nodes $n+1$ (resp. $\overline{n+1}$) and $i$ (resp. $\overline{i}$) for some $1\leq i\leq n$. Note further that if $x$ is a diagram in $\iota(P_n^\mathbb{F}(\delta-1))$ then $\xi_{n+1}x$ (resp. $x\xi_{n+1}$) is equal to $x +\sum_y (\pm) y$  where the $y$'s are $(n+1)$-partition diagrams with a block containing nodes $n+1$ (resp. $\overline{n+1}$) and $i$ (resp. $\overline{i}$) for some $1\leq i\leq n$. Using this, one can easily show that the map $\theta$ is an isomophism of vector spaces. Now one can check that in fact, $\theta$ is an algebra isomorphism by checking the relations between the generators (see for example \cite{halverson2005partition}). For example one can check that $(\xi_{n+1}p_i \xi_{n+1})^2=(\delta -1)\xi_{n+1}p_i \xi_{n+1}$. 

For (ii) we only need that show that 
\begin{eqnarray*}
\xi_{n+1}P_{n+\frac{1}{2}}^\mathbb{F}(\delta) \otimes_{P_{n+\frac{1}{2}}^\mathbb{F}(\delta)} P_{n+\frac{1}{2}}^\mathbb{F}(\delta) \xi_{n+1}&\cong& P_n^\mathbb{F}(\delta),\,\, \mbox{and}\\   P_{n+\frac{1}{2}}^\mathbb{F}(\delta)\xi_{n+1} \otimes_{P_n^\mathbb{F}(\delta -1)} \xi_{n+1} P_{n+\frac{1}{2}}^\mathbb{F}(\delta) &\cong& P_{n+\frac{1}{2}}^\mathbb{F}(\delta).
\end{eqnarray*}
The former follows from (i). For the latter one can show that the multiplication map gives the required isomorphism.
 
The arguments for part (iii) are the same as the ones needed in the proof of (i). We give them here for completeness. Let $\lambda \vdash n-t$ for some $t\geq 0$. First note that for any partition diagram $v\in I(n+\half,n-t+1)$ we have $\xi_{n+1}v=0$ unless $\{n+1, \overline{n+1}\}$ is a block of $v$. This implies that the map
\begin{eqnarray*}
\tau\, : \, V^\bbf(n,n-t) &\longrightarrow& \xi_{n+1} V^\bbf(n+\half, n-t+1)\\
v &\mapsto& \xi_{n+1}\iota(v)
\end{eqnarray*}
is an isomorphism of vector spaces. We will show that, via the algebra isomorphism given in (\ref{algebraiso}), this map is in fact a left $P_n^\mathbb{F}(\delta-1)$-module isomorphism. As $\Delta_\lambda^\bbf(n,\delta - 1)=V^\bbf(n,n-t)\otimes_{\mathfrak{S}_{n-t}}S^{\lambda}_\bbf$ and $\Delta_\lambda^\bbf(n+\half, \delta)=V^\bbf(n+\half, n-t+1)\otimes_{\mathfrak{S}_{n-t}} S^\lambda_\bbf$, this will prove the claim. Thus we need to show that for any partition diagrams in $x\in P_n^\bbf(\delta -1)$ and $v\in  I(n,n-t)$ we have $\tau(xv)=\theta(x)\tau(v)$, that is
$$\xi_{n+1} \iota(xv) = \xi_{n+1}\iota(x)\xi_{n+1}\iota(v).$$
One should note that the product of diagrams $xv$ on the lefthand side takes place in $P_n^\bbf(\delta - 1)$ whereas all the other products are in $P_{n+\half}^\bbf(\delta)$. Clearly it is enough to prove this when $x$ runs over the set of generators $s_{i,j}$, $p_{i,j}$ ($1\leq i<j\leq n$) and $p_i$ ($1\leq i\leq n$) of $P_n^\bbf(\delta - 1)$. If $x=s_{i,j}$ or $p_{i,j}$ then we have that $\iota(x)\xi_{n+1}=\iota(x)+\sum_y (\pm) y$ with $\xi_{n+1}y=0$, so we get $\xi_{n+1}\iota(x)\xi_{n+1}=\xi_{n+1}\iota(x)$. Moreover, multiplication by $x$ does not involve the parameter in this case and so we have $\xi_{n+1}\iota(x)i(v)=\xi_{n+1}\iota(xv)$ as required. We are left with the case $x=p_i$.  Here we have $\xi_{n+1} \iota(p_i) \xi_{n+1} \iota(v)= \xi_{n+1} \iota(p_i) (1-p_{i,n+1})\iota(v) = \xi_{n+1}(\iota (p_i) - \iota(p_i) p_{i,n+1})\iota(v)$. We now consider three cases depending on the block of $v$ containing the node $i$. If $i$ is contained in a propagating block then $\iota(p_i)\iota(v)=\iota(p_iv)$ and $\iota(p_i)p_{i,n+1}\iota(v)=0$ as the concatenation contains fewer than $n-t$ propagating lines.  If $i$ is not contained in a propagating block and is not a singleton block then $\iota(p_i)\iota(v)=\iota(p_iv)$ and $\xi_{n+1}\iota(p_i)p_{i,n+1}\iota(v)=0$ as the node $n+1$ is joined to some other northern node in $\iota(p_i)p_{i,n+1}\iota(v)$. Finally, if the block containing $i$ is simply $\{i\}$ then we have $\iota(p_i)\iota(v)=\delta \iota(v)$ and $\iota(p_i)p_{i,n+1}\iota(v)=\iota(v)$. Therefore we get
$$\xi_{n+1}(\iota(p_i)-\iota(p_i)p_{i,n+1})\iota(v)=\xi_{n+1}(\delta - 1)\iota(v)=\xi_{n+1}\iota(p_i v)$$
as required.
\end{proof}

\section{Ordinary representation theory of the partition algebra}

In this section we recall the results due to P. Martin (see \cite{martin1996structure}) on the representation theory of the partition algebra over a field of characteristic zero, and then reinterpret these in a geometrical setting.

P. Martin showed that the partition algebra $P_n^K(\delta)$ is semisimple if and only if $\delta\notin \{0,1,2,\ldots , 2n-2\}$. In the semisimple case, the simple modules are given by the cell modules $\Delta_\lambda^K(n;\delta)$, and hence are very well-understood.
We will now describe the non-semisimple case. So we will assume that $\delta \in \mathbb{Z}$. We will also assume, as in the previous section, that $\delta \neq 0$. Recall that in this case, the cell-blocks coincide with the blocks of the partition algebra.

\begin{defn}\label{def:deltapair}
	Let $\lambda,\mu$ be partitions, with $\mu\subset\lambda$. We say that $(\mu,\lambda)$ is a \emph{$\delta$-pair}, written $\mu\hookrightarrow_\delta\lambda$, if $\lambda$ differs from $\mu$ by a strip of boxes in a single row, the last of which has content $\delta-|\mu|$.
\end{defn}

Note that if $\delta <0$ then  there are no $\delta$-pairs of partitions.

\begin{example}
We let $\delta=7$, $\lambda=(4,3,1)$ and $\mu=(4,1,1)$. Then we see that $\lambda$ and $\mu$ differ in precisely one row, and the last box in this row of $\lambda$ has content $1$ (see Figure \ref{fig:deltapair}). Since $\delta-|\mu|=7-6=1$, we see that $(\mu,\lambda)$ is a $7$-pair.
\end{example}

\begin{figure}
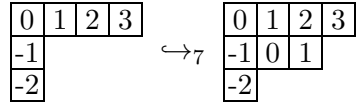

	\centering
	\newcommand{\mone}{\text{-1}}
	\newcommand{\mtwo}{\text{-2}}
	$\Yvcentermath1\young(0123,\mone,\mtwo)~~\hookrightarrow_7~~\young(0123,\mone01,\mtwo)$
	\caption{An example of a $\delta$-pair when $\delta=7$}\label{fig:deltapair}
\end{figure}
Using this definition, the blocks and the composition factors of the cell modules for the partition algebra $P_n^K(\delta)$ can be described as follows.

\begin{thm}[{\cite[Proposition 9]{martin1996structure}}]\label{thm:part0blocks}
Each block of the partition algebra $P^K_n(\delta)$ is given by a chain of partitions
	\[\lambda^{(0)}\subset\lambda^{(1)}\subset\dots\subset\lambda^{(r)}\]
where for each $i$, ($\lambda^{(i)}$,$\lambda^{(i+1)}$) form a $\delta$-pair, differing in the $(i+1)$-th row. Moreover there is an exact sequence of $P_n^K(\delta)$-modules
	\[0\rightarrow\Delta_{\lambda^{(r)}}^K(n)\rightarrow\Delta_{\lambda^{(r-1)}}^K(n)\rightarrow \dots\rightarrow\Delta_{\lambda^{(1)}}^K(n)\rightarrow\Delta_{\lambda^{(0)}}^K(n)\rightarrow L_{\lambda^{(0)}}^K(n)\rightarrow0\]
	with the image of each homomorphism a simple module. In particular, each of the cell modules $\Delta_{\lambda^{(i)}}^K(n)$ for $0\leq i<r$ has Loewy structure
	\begin{center}
		$L_{\lambda^{(i)}}^K(n)$\\
		$L_{\lambda^{(i+1)}}^K(n)$
	\end{center}
	and $\Delta_{\lambda^{(r)}}^K(n)=L_{\lambda^{(r)}}^K(n)$.
\end{thm}

\subsection*{Reflection geometry}
We now give an equivalent description of the blocks of the partition algebra $P_n^K(\delta)$ in terms of a reflection group action, similar to the one given for blocks of the symmetric group in positive characteristic in Section 3. We will see at the end of Section 8, how these two combine in some sense to give the blocks of the partition algebra in positive characteristic.

Let $\{\varepsilon_0,\dots,\varepsilon_n\}$ be a set of formal symbols and set
	\[\widehat{E}_n=\bigoplus_{i=0}^n\mathbb{R}\varepsilon_i\]
to be the $(n+1)$-dimensional space with basis $\varepsilon_0, \ldots , \varepsilon_{n}$.
We have an inner product $\langle~,~\rangle$ on $\widehat{E}_n$ given by extending linearly the relations
	\[\langle\varepsilon_i,\varepsilon_j\rangle=\delta_{ij}\]
for $0\leq i,j\leq n$ where $\delta_{ij}$ is the Kronecker delta. 

Let $\widehat{\Phi}_n=\{\varepsilon_i-\varepsilon_j:0\leq i,j\leq n\}$ be a root system of type $A_n$, and $\widehat{W}_n\cong\sym_{n+1}$ the corresponding Weyl group, generated by the reflections $s_{i,j}$ ($0\leq i<j\leq n$) defined by
	\[s_{i,j}(x)=x-\langle x,\varepsilon_i-\varepsilon_j\rangle(\varepsilon_i-\varepsilon_j)\]
	for all $x\in\widehat{E}_n$.

Observe that we have extended the Euclidian space $E_n$ from Section 3 by adding the basis vector $\varepsilon_0$, resulting in a slightly unusual labelling of the roots. This is to ensure consistency when we combine these results in Section 8.  Note also that, as opposed to  Section 3,  we do not consider affine reflections in this case.

	Now we also extend the shift $\rho_n$ and define $\rho_n(\delta)=(\delta,-1,-2,\dots,-n)\in \widehat{E}_n$. We then define a shifted action of $\widehat{W}_n$ on $\widehat{E}_n$, given by
	\[w\cdot_\delta x=w(x+\rho_n(\delta))-\rho_n(\delta)\]
for all $w\in \widehat{W}_n$ and $x\in \widehat{E}_n$.
Given a partition $\lambda=(\lambda_1,\lambda_2,\dots)\in \Lambda_{\leq n}$, let 
\[\hat\lambda=(-|\lambda|,\lambda_1,\dots,\lambda_n)=-|\lambda|\varepsilon_0+\sum_{i=1}^n\lambda_i\varepsilon_i\in \widehat{E}_n.\]
We then have the following reformulation of the blocks of $P_n^K(\delta)$.
\begin{thm}\label{thm:part0blocksgeom}
	Let $\lambda,\mu\in\Lambda_{\leq n}$.  Then we have $\mu\in \mathcal{B}_\lambda (P_n^K(\delta))$ if and only if $\hat\mu\in \widehat{W}_n\cdot_\delta\hat\lambda$.
\end{thm}
\begin{proof}
	We saw in Theorem 5.3 that the blocks of $P_n^K(\delta)$ are given by maximal chains of partitions
		\[\lambda^{(0)}\subset\lambda^{(1)}\subset\dots\subset\lambda^{(r)}\]
	where for each $i$, ($\lambda^{(i-1)}$,$\lambda^{(i)}$) form a $\delta$-pair, differing in the $i$-th row. We claim that $\widehat{\lambda^{(i)}}=s_{0,i}\cdot_\delta\widehat{\lambda^{(i-1)}}$. Indeed,
	\begin{align}
		s_{0,i}\cdot_\delta\widehat{\lambda^{(i-1)}}&=(\lambda_i^{(i-1)}-i,\lambda^{(i-1)}_1-1,\dots,-|\lambda^{(i-1)}|+\delta,\dots,\lambda^{(i-1)}_n-n)-\rho(\delta)\nonumber\\
			&=(\lambda_i^{(i-1)}-i-\delta,\lambda^{(i-1)}_1,\lambda^{(i-1)}_2,\dots,-|\lambda^{(i-1)}|+\delta+i,\dots,\lambda^{(i-1)}_n).\label{eq:partchain}
	\end{align}
	Now the partition 
	\[(\lambda^{(i-1)}_1,\lambda^{(i-1)}_2,\dots,-|\lambda^{(i-1)}|+\delta+i,\dots,\lambda^{(i-1)}_n)\]
	obtained from \eqref{eq:partchain} differs from $\lambda^{(i-1)}$ by a strip of boxes in row $i$ only, the last of which has content
		\[(-|\lambda^{(i-1)}|+\delta+i)-i=\delta-|\lambda^{(i-1)}|\]
	and so $s_{0,i}\cdot_\delta\widehat{\lambda^{(i-1)}}=\widehat{\lambda^{(i)}}$ as claimed. Therefore if $\mu\neq\nu\in\Lambda_{\leq n}$ are in the same block then $\mu=\lambda^{(i)}$ and $\nu=\lambda^{(j)}$ for some $i<j$ say, and
		\[\hat\nu=(s_{0,j}\dots s_{0,i+2}s_{0,i+1})\cdot_\delta\hat\mu.\]

	Conversely, suppose $\lambda,\mu\in\Lambda_{\leq n}$ satisfy $\hat\mu\in \widehat{W}_n\cdot_\delta\hat\lambda$. Since $\lambda$ is a partition, the sequence $(\lambda_1-1,\lambda_2-2,\dots,\lambda_n-n)$ is strictly decreasing, and similarly for $\mu$. Therefore if $\hat\mu\in \widehat{W}_n\cdot_\delta\hat\lambda$, then $\hat\mu+\rho_n(\delta)=w(\hat\lambda+\rho_n(\delta))$ for some $w\in \widehat{W}_n$ not fixing entry $0$, and we have
		\[ \hat\mu+\rho_n(\delta)=(\lambda_i-i,\dots)\]
	for some $1\leq i\leq n$. If $\lambda_i-i=\delta-|\lambda|$ then $\mu=\lambda$ and the result is immediate.
	
\noindent If now $\lambda_i-i<\delta-|\lambda|$, then
		\[\hat\mu+\rho_n(\delta)=(\lambda_i-i,\dots,\lambda_j-j,\delta-|\lambda|,\lambda_{j+1}-(j+1),\dots,\lambda_{i-1}-(i-1),\lambda_{i+1}-(i+1),\dots)\]
	for some $j$. If instead $\lambda_i-i>\delta-|\lambda|$ then
		\begin{align*}
		\hat\mu&+\rho_n(\delta)=\\&(\lambda_i-i,\dots,\dots,\lambda_{i-1}-(i-1),\lambda_{i+1}-(i+1),\dots,\lambda_j-j,\delta-|\lambda|,\lambda_{j+1}-(j+1),\dots)\end{align*}
	for some $j$. In either case, we have
		\[\hat\mu+\rho_n(\delta)=(s_{0,i}\dots s_{0,j+2}s_{0,j+1})\cdot_\delta(\hat\lambda+\rho_n(\delta)).\]
Using the calculation in \eqref{eq:partchain} we see that $\lambda$ and $\mu$ must be elements in a chain of $\delta$-pairs, and so are in the same block.
\end{proof}

\section{A necessary condition for blocks}

Now we turn to the representation theory of the partition algebra $P_n^k(\delta)$ over the field $k$ of positive characteristic $p>0$. We assume that $\delta\in k$ is non-zero.
In this section we will use the action of the Jucys-Murphy elements on the cell-modules to deduce a necessary condition for two partitions to be in the same cell-block.

 The Jucys-Murphy elements for the partition algebra were introduced in \cite{halverson2005partition}. These elements were later defined inductively in \cite[Section 2.3]{enyang2013seminormal} as follows.

\begin{defn}
\begin{itemize}
	\item[(i)] Set $L_0=0$, $L_1=p_1$, $\sigma_1=1$, $\sigma_2=s_1$ and for $i\geq1$, define
		\[L_{i+1}=-s_iL_ip_{i+\half}-p_{i+\half}L_is_i+p_{i+\half}L_ip_{i+1}p_{i+\half}+s_iL_is_i+\sigma_{i+1},\]
where for $i\geq2$ we define
			\begin{align*}	\sigma_{i+1}=s_{i-1}s_i\sigma_i&s_is_{i+1}+s_ip_{i-\half}L_{i-1}s_ip_{i-\half}+p_{i-\half}L_{i-1}s_ip_{i-\half}\\&-s_ip_{i-\half}L_{i-1}s_{i-1}p_{i+\half}p_ip_{i-\half}-p_{i-\half}p_ip_{i+\half}s_{i-1}L_{i-1}p_{i-\half}s_i.\end{align*}
		
	\item[(ii)] Set $L_{\half}=0$, $\sigma_\half=1$, $\sigma_{1+\half}=1$ and for $i\geq1$, define
\[L_{i+\half} = -L_ip_{i+\half} - p_{i+\half} L_i + p_{i+\half}L_ip_ip_{i+\half}+s_iL_{i-\half}s_i+\sigma_{i+\half},\]
where for $i\geq 2$ we define
		\begin{align*}
			\sigma_{i+\half}=s_{i-1}s_i&\sigma_{i-\half} s_is_{i-1}+p_{i-\half}L_{i-1}s_ip_{i-\half}s_i+s_ip_{i-\half}L_{i-1}s_ip_{i-\half}\\
			&-p_{i-\half}L_{i-1}s_{i-1}p_{i+\half}p_ip_{i-\half}-s_ip_{i-\half}p_ip_{i+\half}s_{i-1}L_{i-1}p_{i-\half}s_i.
		\end{align*}
\end{itemize}
\end{defn}
If we project these elements onto the quotient $P^k_n(\delta)/J^{(n-1)}_n$, where $J^{(n-1)}_n$ is defined in \eqref{eq:partfiltration}, then we obtain the following result.
\begin{lem}\label{lem:seminormal}
	\begin{itemize}
		\item[(i)] $\sigma_i+J^{(n-1)}_n=s_{i-1}+J^{(n-1)}_n$ for all $i\geq 2$.
		\item[(ii)] $L_i+J^{(n-1)}_n=\sum_{j=1}^{i-1}s_{j,i}+J^{(n-1)}_n$ for all $i\geq2$.
		\item[(iii)] $\sigma_{i+\half}+J^{(n-1)}_n=1+J^{(n-1)}_n$ for all $i\geq0$.
		\item[(iv)] $L_{i+\half}+J^{(n-1)}_n= i+J^{(n-1)}_n$ for all $i\geq0$.
		\item[(v)] Let $Z_n=L_\half+L_1+L_{1+\half}+\dots+L_n$. Then
					\[Z_n+J^{(n-1)}_n=\frac{n(n-1)}{2}+\sum_{1\leq i<j\leq n}s_{i,j}+J^{(n-1)}_n.\]
	\end{itemize}
\end{lem}
	\begin{proof}
		We prove these statements by induction on $i$.
		\begin{itemize}
		\item[(i)] This is true for $i=2$ by definition. Now let $i\geq2$, then we have
			\begin{align*}
				\sigma_{i+1}+J^{(n-1)}_n&= s_{i-1}s_i\sigma_is_is_{i-1}+J^{(n-1)}_n\\
						&=s_{i-1}s_is_{i-1}s_is_{i-1}+J^{(n-1)}_n\text{ by induction}\\
						&=s_i+J^{(n-1)}_n.
			\end{align*}
		\item[(ii)] We have $L_2+J^{(n-1)}_n=\sigma_2+J^{(n-1)}_n= s_1+J^{(n-1)}_n$. Now let $i\geq2$, then we have
			\begin{align*}
				L_{i+1}+J^{(n-1)}_n&= s_iL_is_i+\sigma_{i+1}+J^{(n-1)}_n\\
					&= s_i\left(\sum_{j=1}^{i-1}s_{j,i}\right)s_i+s_i+J^{(n-1)}_n\text{ by induction and using (i)}\\
					&= \sum_{j=1}^{i-1}s_{j,i+1}+s_i+J^{(n-1)}_n\\
					&=\sum_{j=1}^is_{j,i+1}+J^{(n-1)}_n.
			\end{align*}
		\item[(iii)] We have $\sigma_{\half}=1$, and for $i\geq1$
			\begin{align*}
				\sigma_{i+\half}+J^{(n-1)}_n&= s_{i-1}s_i\sigma_{i-\half}s_is_{i-1}+J^{(n-1)}_n\\
							&= s_{i-1}s_i1s_is_{i-1}+J^{(n-1)}_n\\
							&= 1+J^{(n-1)}_n.
			\end{align*}
		\item[(iv)] We have $L_\half=0$, and for $i\geq1$
			\begin{align*}
				L_{i+\half}+J^{(n-1)}_n&= s_iL_{i-\half}s_i+\sigma_{i+\half}+J^{(n-1)}_n\\
						&= s_i(i-1)s_i+1+J^{(n-1)}_n\text{ by induction and using (iii)}\\
						&= i+J^{(n-1)}_n.
			\end{align*}
		\item[(v)] Follows immediately from (ii) and (iv).
		\end{itemize}
	\end{proof}
\noindent Recall the following result.
\begin{lem}[{\cite[Theorem 3.35]{halverson2005partition}, \cite[Lemma 3.14]{enyang2013seminormal}}]\label{lem:partzn}
	Let $\mu\in\Lambda_{\leq n}$ with $|\mu|=n-t$ for some $t\geq0$. Then $Z_n$ acts on $\Delta_\mu^k(n;\delta)$ as scalar multiplication by
	\[t\delta+\left(\!\begin{array}{c}|\mu|\\2\end{array}\!\right)+\mathrm{ct}(\mu).\]
\end{lem}

\noindent We can now prove the following theorem. Note that this result was proved over the field of complex numbers $\mathbb{C}$ in \cite[Theorem 6.1]{doran2000partition}. However, their proof does not generalise to arbitrary fields.

\begin{thm}\label{thm:partdeltapblocks}
	Let $\lambda,\mu\in\Lambda_{\leq n}$ with $\lambda\in \Lambda_{\leq n}^*$. If $[\Delta^k_\mu(n;\delta):L^k_\lambda(n;\delta)]\neq 0$ then $|\lambda|-|\mu|=t\geq 0$ and we have
\begin{equation}\label{eq:deltappair}
t\delta - t|\mu| -\mathrm{ct}(\lambda)+\mathrm{ct}(\mu) -\frac{t(t-1)}{2} = 0
\end{equation}
in the field $k$. 
\end{thm}

 \noindent Note that $\frac{t(t-1)}{2}$ and $\binom{ |\mu|}{2} $ are both integers and so the expressions in Lemma 6.3 and Theorem 6.4 make  sense in the field $k$.

	\begin{proof} The fact that $t\geq 0$ follows directly from the cellularity of $P_n^k(\delta)$.
		Now, by use of the localisation functors given in \eqref{eq:partfn} we may assume that $\lambda\vdash n$ and $\mu\vdash n-t$. Therefore we have $\Delta_\lambda^k(n;\delta)\cong S^\lambda_k$, and the ideal $J_n^{(n-1)}$ acts as zero on $\Delta_\lambda^k(n;\delta)$.
		
		Now, as $[\Delta^k_\mu(n;\delta):L^k_\lambda(n;\delta)]\neq 0$ and $L_\lambda^k(n;\delta)$ appears as the head of $\Delta_\lambda^k(n;\delta)$,   there exist  submodules $M\subset N\subset \Delta_\mu^k(n;\delta)$ and a non-zero homomorphism
			\[\Delta_\lambda^k(n;\delta)\cong S^\lambda_k\longrightarrow N/M.\]
		By Lemma \ref{lem:seminormal}(v), the element
			\begin{equation}Z_n-\frac{n(n-1)}{2}-\sum_{1\leq i<j\leq n}s_{i,j}\label{eq:partzn}\end{equation}
		must act as zero on $N/M$. It is well-known that $\sum_{1\leq i<j\leq n}s_{i,j}$ acts by the scalar $\mathrm{ct}(\lambda)$ on $S^\lambda_k$ (see for example \cite[Chapter 1, Section 1 Exercise 3 and Section 7 Example 7]{macdonald}), and hence also on $N/M$. Using Lemma \ref{lem:partzn}, we then see that the element given in \eqref{eq:partzn} acts on $N/M$ by the scalar
		\begin{align*}
			&t\delta+\left(\!\begin{array}{c}|\mu|\\2\end{array}\!\right)+\mathrm{ct}(\mu)-\frac{n(n-1)}{2}-\mathrm{ct}(\lambda)\\
			=&t\delta-t|\mu|-\mathrm{ct}(\lambda)+\mathrm{ct}(\mu)-\frac{t(t-1)}{2}.
		\end{align*}
		Hence this must be zero in the field $k$. 
	\end{proof}

We will now strengthen this result to obtain a necessary condition for the cell-blocks of the partition algebra.
Let us start with the following lemma.

\begin{lem}\label{lem:surjhom}
 Let $\lambda\in \Lambda_n^*$ be a $p$-regular partition of $n>1$. Then there exists a removable node $\varepsilon_i$ of $\lambda$ such that $\lambda - \varepsilon_i$ is $p$-regular and there is a surjective homomorphism
$${\rm ind}_{n-\half}\Delta_{\lambda - \varepsilon_i}^k(n-\half;\delta) \longrightarrow \Delta_\lambda^k(n;\delta).$$
\end{lem}

\begin{proof}
As $\lambda$ is a partition of $n>1$, it has at least one removable node. Now choose the removable node $\varepsilon_i$, in row $i$, with $i$ minimal such that $\lambda - \varepsilon_i$ is $p$-regular. Now consider the set $\Gamma$ of all partitions $\mu$ with $\mu\triangleright \lambda - \varepsilon_i$. Clearly we have $\lambda \in \Gamma$ as $\lambda = \lambda - \varepsilon_i + \varepsilon_i$. Now, using (\ref{eq:indhalf}), we have a surjective homomorphism
\begin{equation}\label{eq:indhalfextra}
{\rm ind}_{n-\half}\Delta_{\lambda - \varepsilon_i}^k(n-\half;\delta) \longrightarrow \biguplus_{\mu\in \Gamma}\Delta_\mu^k(n;\delta)
\end{equation}
where the factors in the module $\biguplus_{\mu\in \Gamma}\Delta_\mu^k(n;\delta)$ can be ordered by dominance starting with the most dominant at the top. Note that as $\lambda$ is a partition of $n$, so is every $\mu\in \Gamma$ and hence $\Delta_\mu^k(n;\delta)=S_k^\mu$. Thus the module on the righthand side of (\ref{eq:indhalfextra}) is a $k\sym_n$-module, trivially inflated to $P_n^k(\delta)$. Now this module will decompose according to the block structure of $k\sym_n$. We claim that $\lambda$ is the most dominant partition in $\Gamma \cap \mathcal{B}_\lambda(k\sym_n)$. This will imply that $\Delta_\lambda^k(n;\delta)$ appears as a quotient of $\biguplus_{\mu\in \Gamma}\Delta_\mu^k(n;\delta)$, and hence we get the required homomorphism by composing the map given in (\ref{eq:indhalfextra}) with the projection onto that quotient. It remains to prove that $\lambda$ is indeed the most dominant partition in $\Gamma \cap \mathcal{B}_\lambda(k\sym_n)$. If $i=1$ we are done. Now suppose that $i>1$.
By our choice of row $i$, if we remove any node of $\lambda$ in an earlier row, say $j<i$, then the resulting partition $\lambda - \varepsilon_j$ is not $p$-regular. Using this, and the fact that $\lambda$ is $p$-regular, we deduce that $\lambda$ and $\lambda - \varepsilon_i$ must be of the form
\begin{align*}
  \lambda &= (b^a, (b-1)^{p-1}, (b-2)^{p-1}, \ldots , (b-t+1)^{p-1}, (b-t)^{p-1}, \ldots), \,\, \mbox{and}\\
  \lambda - \varepsilon_i &= (b^a, (b-1)^{p-1}, (b-2)^{p-1}, \ldots , (b-t+1)^{p-1}, (b-t)^{p-2}, b-t-1, \ldots)
\intertext{
for some $1<a<p$, $b>0$ and $t\geq 0$. Now the partitions $\mu\in \Gamma$ with $\mu \succ \lambda$ are precisely the partitions}
  \mu^{(0)} &= (b+1, b^{a-1}, (b-1)^{p-1}, (b-2)^{p-1}, \ldots , (b-t+1)^{p-1}, (b-t)^{p-2},b-t-1, \ldots),\\
  \mu^{(1)} &= (b^{a+1}, (b-1)^{p-2}, (b-2)^{p-1}, \ldots , (b-t+1)^{p-1}, (b-t)^{p-2}, b-t-1, \ldots),\\
  \mu^{(2)} &= (b^a, (b-1)^p, (b-2)^{p-2},  \ldots , (b-t+1)^{p-1}, (b-t)^{p-2}, b-t-1,\ldots),\\ 
  \ldots & \\
  \mu^{(t)} &=  (b^a, (b-1)^{p-1}, (b-2)^{p-1}, \ldots , (b-t+1)^{p}, (b-t)^{p-3}, b-t-1, \ldots).
\end{align*}
Now it is easy to see from Theorem \ref{thm:nakayama} (or the reformulation given in Theorem \ref{newnakayama}) that none of these partitions belong to $\mathcal{B}_\lambda(k\sym_n)$.
\end{proof}

For any $\lambda\in \Lambda_{\leq n}$, recall that we denote by $\hat{\lambda}$ the $(n+1)$-tuple
$$\hat{\lambda}=(-|\lambda|, \lambda_1, \lambda_2, \ldots, \lambda_n)$$
where $\lambda_i=0$ for $l(\lambda)<i\leq n$. For $\delta\in k$ we also define the $(n+1)$-tuple
$$\rho_n(\delta)=(\delta , -1, -2, -3, \ldots , -n).$$
Note that both $\hat{\lambda}$ and $\rho_n(\delta)$ can be viewed as $(n+1)$-tuples of elements in $k$. For any such $(n+1)$-tuples $x$ and $y$ of elements of $k$, say $x=(x_0, x_1, \ldots , x_n)$ and $y=(y_0, y_1, \ldots , y_n)$, we write
$$x\sim_k y$$
if and only if there exists a permutation $\sigma\in \sym_{n+1}$ such that $x_i = y_{\sigma(i)}$ for all $0\leq i\leq n$.
With this definition, we can now state the following theorem.

\begin{thm}\label{thm:blocks<orbits}
	Let $\lambda,\mu\in\Lambda_{\leq n}$ with $\lambda\in \Lambda_{\leq n}^*$.  If $[\Delta_\mu^k(n;\delta) : L_\lambda^k(n;\delta)]\neq 0$, then $\hat{\mu}+\rho_n(\delta) \sim_k \hat{\lambda} + \rho_n(\delta)$.
\end{thm}

	\begin{proof}
Note that, as in the proof of Theorem \ref{thm:partdeltapblocks}, if $[\Delta_\mu^k(n;\delta):L_\lambda^k(n;\delta)]\neq 0$ then there exists a submodule $M$ of $\Delta_\mu^k(n;\delta)$ and a non-zero homomorphism
\begin{equation}\label{nonzerohom1}
\Delta_\lambda^k(n;\delta) \longrightarrow \Delta_\mu^k(n;\delta)/M.
\end{equation}

		By use of the localisation functors given in \eqref{eq:partfn} and cellularity   we may assume that $|\lambda|= n$ and $|\mu|= n-t$ for some $t\geq0$. We prove the result by induction on $n$.
		
		If $n=0$ there is nothing to prove, so assume $n\geq1$. If $\lambda=\emptyset$  we must also have $\mu=\emptyset$, and the result holds trivially.
		
		If now $|\lambda|\geq1$, then $\lambda$ has a removable node. Pick the removable node satisfying the conditions of Lemma \ref{lem:surjhom}. Then we have a surjective homomorphism
\begin{equation}\label{nonzerohom2}
			\mathrm{ind}_{n-\half}\Delta_{\lambda-\varepsilon_i}^k(n-\textstyle\half;\delta)\longrightarrow\Delta_\lambda^k(n;\delta).
\end{equation}
Composing the homomorphisms (\ref{nonzerohom1}) and (\ref{nonzerohom2}) we obtain a non-zero homomorphism
$$\mathrm{ind}_{n-\half}\Delta_{\lambda-\varepsilon_i}^k(n-\textstyle\half;\delta)\longrightarrow  \Delta_\mu^k(n;\delta)/M.$$
Now by Frobenius reciprocity we have
		\begin{align*}
			&\mathrm{Hom}(\mathrm{ind}_{n-\half}\Delta_{\lambda-\varepsilon_i}^k(n-\textstyle\half;\delta),\Delta_\mu^k(n;\delta)/M) 
			\cong \mathrm{Hom}(\Delta_{\lambda-\varepsilon_i}^k(n-\textstyle\half;\delta),\mathrm{res}_{n}(\Delta_\mu^k(n;\delta)/M))\neq0.
		\end{align*}
		Using  the branching rule \eqref{eq:resn} we have either
			\[\mathrm{Hom}(\Delta_{\lambda-\varepsilon_i}^k(n-\textstyle\half;\delta),\Delta_\mu^k(n-\half;\delta)/N)\neq0\]
		for some submodule $N\subset\Delta_\mu^k(n-\half;\delta)$, or
			\[\mathrm{Hom}(\Delta_{\lambda-\varepsilon_i}^k(n-\textstyle\half;\delta),\Delta_{\mu-\varepsilon_j}^k(n-\half;\delta)/Q)\neq0\]
		for some removable node $\varepsilon_j$ in row $j$ of $\mu$ say, and some submodule $Q\subset\Delta_{\mu-\varepsilon_j}^k(n-\half;\delta)$. \\
		Applying Theorem \ref{prop:morita} we have the following two cases:
		\begin{description}\setlength{\itemindent}{0cm}
			\item[Case 1] $\mathrm{Hom}(\Delta_{\lambda-\varepsilon_i}^k(n-1;\delta-1),\Delta_\mu^k(n-1;\delta-1)/N)\neq0$ for some submodule $N\subset\Delta_\mu^k(n-1;\delta-1)$, and so $[\Delta_\mu^k(n-1;\delta -1):L_{\lambda - \varepsilon_i}^k(n-1;\delta -1)]\neq 0$, or
			\item[Case 2] $\mathrm{Hom}(\Delta_{\lambda-\varepsilon_i}^k(n-1;\delta-1),\Delta_{\mu-\varepsilon_j}^k(n-1;\delta-1)/Q)\neq0$ for some removable node $\varepsilon_j$ in row $j$ of $\mu$, and some submodule $Q\subset\Delta_{\mu-\varepsilon_j}^k(n-1;\delta-1)$, and so $[\Delta_{\mu - \varepsilon_j}^k(n-1;\delta -1):L_{\lambda - \varepsilon_i}^k(n-1;\delta -1)]\neq 0$.
		\end{description}
We consider each case in turn.\\
		\textbf{Case 1}~~Applying our inductive step, we have that  $\hat\mu+\rho_{n-1}(\delta-1)\sim_p\widehat{\lambda-\varepsilon_i}+\rho_{n-1}(\delta-1)$, that is
			\begin{equation}(\delta-1-|\mu|,\mu_1-1,\dots,\mu_{n-1}-n+1)\sim_k(\delta-|\lambda|,\lambda_1-1,\dots,\lambda_i-i-1,\dots,\lambda_{n-1}-n+1). \label{eq:partcase1seq}\end{equation}
	As $|\lambda-\varepsilon_i|-|\mu|=t-1$, we also know from Theorem \ref{thm:partdeltapblocks} that $\lambda-\varepsilon_i$ and $\mu$ satisfy 
		\[(t-1)(\delta-1)-(t-1)|\mu|-\mathrm{ct}(\lambda)+\mathrm{ct}(\varepsilon_i)+\mathrm{ct}(\mu)-\frac{(t-1)(t-2)}{2}=0\]
over the field $k$.
		Hence we can deduce that 
\begin{equation}\label{deltapair1}
t\delta-t|\mu|-\mathrm{ct}(\lambda)+\mathrm{ct}(\mu)-\frac{t(t-1)}{2}+\mathrm{ct}(\varepsilon_i)+|\mu|-\delta = 0
\end{equation}
in the field $k$.
		Moreover by assumption and by Theorem \ref{thm:partdeltapblocks} we have that $\lambda$ and $\mu$ satisfy
\begin{equation}\label{deltapair2}
t\delta - t|\mu| -\mathrm{ct}(\lambda)+\mathrm{ct}(\mu) -\frac{t(t-1)}{2} = 0.
\end{equation}
It follows from equation (\ref{deltapair1}) and (\ref{deltapair2}) that 

		\begin{equation}\mathrm{ct}(\varepsilon_i)=\lambda_i-i = \delta-|\mu|\label{eq:partcase1cong}\end{equation}
in the field $k$.

		Combining \eqref{eq:partcase1seq} and \eqref{eq:partcase1cong}, the sequences 
		\[\hat\lambda+\rho_n(\delta)=(\delta-|\lambda|,\lambda_1-1,\dots,\lambda_i-i,\dots,\lambda_n-n)\]
		and
		\[\hat\mu+\rho_n(\delta)=(\delta-|\mu|,\mu_1-1,\dots,\mu_n-n)\]
		satisfy $\hat{\lambda} + \rho_n(\delta) \sim_k \hat{\mu} + \rho_n(\delta)$ as required.
		\\\\
		\textbf{Case 2}~~Applying our inductive step, we have that  $\widehat{\mu-\varepsilon_j}+\rho_{n-1}(\delta-1)\sim_k\widehat{\lambda-\varepsilon_i}+\rho_{n-1}(\delta-1)$, that is
			\begin{equation}(\delta-|\mu|,\mu_1-1,\dots,\mu_j-j-1,\dots,\mu_{n-1}-n+1)\sim_k(\delta-|\lambda|,\lambda_1-1,\dots,\lambda_i-i-1,\dots,\lambda_{n-1}-n+1).\label{eq:partcase2seq}\end{equation}
		Since $|\lambda-\varepsilon_i|-|\mu-\varepsilon_j|=t$, we also know from Theorem \ref{thm:partdeltapblocks} 
		\[t(\delta-1)-t(|\mu|-1)-\mathrm{ct}(\lambda)+\mathrm{ct}(\varepsilon_i)+\mathrm{ct}(\mu)-\mathrm{ct}(\varepsilon_j)-\frac{t(t-1)}{2}=0\]
in the field $k$.
		Moreover by assumption and Theorem \ref{thm:partdeltapblocks}, we have that $\lambda$ and $\mu$ satisfy (\ref{deltapair2}). It follows that 
		\[\mathrm{ct}(\varepsilon_i) = \mathrm{ct}(\varepsilon_j),\]
		that is,
		\begin{equation}\lambda_i-i = \mu_j-j\label{eq:partcase2cong}\end{equation}
in the field $k$.
		Combining \eqref{eq:partcase2seq} and \eqref{eq:partcase2cong}, we get that the sequences 
		\[\hat\lambda+\rho_n(\delta)=(\delta-|\lambda|,\lambda_1-1,\dots,\lambda_i-i,\dots,\lambda_n-n)\]
		and
		\[\hat\mu+\rho_n(\delta)=(\delta-|\mu|,\mu_1-1,\dots,\mu_j-j,\dots,\mu_n-n)\]
		satisfy $\hat{\lambda} +\rho_n(\delta) \sim_k \hat{\mu} + \rho_n(\delta)$ as required.
	\end{proof}
Since the cell-blocks of $P_n^k(\delta)$ are defined as the equivalence classes of the equivalence relation on $\Lambda_{\leq n}$ generated by $[\Delta_\mu^k(n;\delta):L_\lambda^k(n;\delta)]\neq 0$ we immediately obtain the following corollary.
\begin{cor}\label{cor:blocks<orbits}
	Let $\lambda,\mu\in\Lambda_{\leq n}$. If $\mu\in\mathcal{B}_\lambda(P_n^k(\delta))$, then $\hat{\mu}+\rho_n(\delta) \sim_k \hat{\lambda}+\rho_n(\delta).$

\end{cor}

\section{Blocks for non-integer parameter $\delta$}

In this section, we obtain a description of the blocks of $P^k_n(\delta)$ when $\delta$ does not belong to the prime subfield $\mathbb{F}_p\subset k$. First we recall the following result which holds for an arbitrary $\delta\in k$.

\begin{thm}[{\cite[Corollary 6.2]{hartmann2010cohomological}}] \label{thm:decompsamesize}
		Let $\lambda,\mu\in \Lambda_{\leq n}$ with $\lambda\in \Lambda_{\leq n}^*$.  If $|\lambda|=|\mu|=n-t$ then 
			\[[\Delta_\mu^k(n;\delta):L_\lambda^k(n;\delta)]=[S^\mu_k:D^\lambda_k]\]
where $S^\mu_k$ and $D_k^\lambda$ denote the Specht and simple modules respectively for the symmetric group algebra $k\sym_{n-t}$.
\end{thm}

Using Corollary \ref{cor:blocks<orbits} and Theorem \ref{thm:decompsamesize} we can now deduce the following theorem.

\begin{thm} \label{blocksnoninteger}
Assume $\delta\notin \mathbb{F}_p$. Let $\lambda, \mu\in \Lambda_{\leq n}$ then the following propositions are equivalent.
\begin{enumerate}
\item $\mu\in \mathcal{B}_\lambda(P^k_n(\delta))$.
\item $|\mu|=|\lambda|=n-t$ for some $t\geq 0$ and $\mu\in \mathcal{B}_\lambda(k\sym_{n-t})$.
\item $|\mu|=|\lambda|=n-t$ for some $t\geq 0$ and $\lambda + \rho_{n-t} \sim_p \mu + \rho_{n-t}$.
\end{enumerate}
\end{thm}

\begin{proof}
We have already seen in Theorem \ref{newnakayama} that (2) and (3) are equivalent. We will now show that (1) and (2) are equivalent. From Theorem \ref{thm:decompsamesize} we have that (2) implies (1). We will now show that (1) implies (2). In fact, using Theorem \ref{thm:decompsamesize}, all we need to show is that if $\mu\in \mathcal{B}_\lambda(P_n^k(\delta))$ then $|\mu|=|\lambda|$.

Let $\lambda, \mu\in \Lambda_{\leq n}$ with $\lambda\in \Lambda_{\leq n}^*$ satisfying $[\Delta_\mu^k(n;\delta):L_\lambda^k(n,\delta)]\neq 0$.  We can assume that $\lambda \vdash n$ and $\mu\vdash n-t$ for some $t\geq 0$. We know from Corollary \ref{cor:blocks<orbits} that $\hat{\lambda}+\rho_n(\delta) \sim_k \hat{\mu} +\rho_n(\delta)$. Now, as $\delta\notin \mathbb{F}_p$ we must have $|\lambda|=|\mu| +sp$ for some $s\geq 0$. We will show by induction on $n$ that we must have $s=0$.

If $n=1$ then there is nothing to prove as $|\lambda|, |\mu|\leq 1$.

Now assume that $n>1$, and suppose for a contradiction that $s>0$. Following the same argument as in the proof of Theorem \ref{thm:blocks<orbits} we know that $\lambda$ has a removable node $\varepsilon_i$ with $\lambda -\varepsilon_i \in \Lambda_{\leq n-1}^*$ satisfying either

\textbf{Case 1:} $[\Delta_\mu^k(n-1;\delta -1):L_{\lambda - \varepsilon_i}^k(n-1;\delta -1)]\neq 0$, or

\textbf{Case 2:}  $[\Delta_{\mu-\varepsilon_j}^k(n-1;\delta -1):L_{\lambda - \varepsilon_i}^k(n-1;\delta -1)]\neq 0$, for some removable node $\varepsilon_j$.

Note that $\delta - 1\notin \mathbb{F}_p$. Moreover, as $|\lambda - \varepsilon_i|-|\mu| = sp-1$ we see that Case 1 is impossible. Now from Case 2 we obtain, by induction that $|\lambda - \varepsilon_i|=|\mu - \varepsilon_j|$, and so $|\lambda|=|\mu|$ as required.
\end{proof}

\section{Blocks for integer parameter $\delta$}

In Section 7, we obtained a complete description of the (cell-)blocks of $P_n^k(\delta)$ when $\delta\notin \mathbb{F}_p$.
We will now consider the (cell-)blocks when $\delta \in \mathbb{F}_p$ with $\delta\neq 0$. Note that Corollary \ref{cor:blocks<orbits} holds for any $\delta \in k$ (with $\delta\neq 0$) and provides a necessary condition for cell-blocks. We will show that in fact, when $\delta \in \mathbb{F}_p$, this is also a sufficient condition. We now make this more precise.

Fix $\delta \in \mathbb{F}_p\subset k$ with $\delta\neq 0$. Then for any partition $\lambda$, the element $\hat{\lambda}+\rho_n(\delta)$ can be viewed as a sequence in $\mathbb{F}_p$. For such sequences we will write $\sim_p$  instead of $\sim_k$ as this relation only depends on the prime number $p$, recovering the definition introduced in Section 3.

\begin{defn} For $\lambda \in \Lambda_{\leq n}$, we define $\mathcal{O}_\lambda^p(n;\delta)$ to be the set of all $\mu\in \Lambda_{\leq n}$ satisfying $\hat{\mu}+\rho_n(\delta) \sim_p \hat{\lambda}+\rho_n(\delta)$.
\end{defn}

Using this definition, Corollary \ref{cor:blocks<orbits} can be rephrased as follows: If $\mu\in \mathcal{B}_\lambda(P_n^k(\delta))$ then $\mu\in \mathcal{O}^p_\lambda(n;\delta)$.
In this section we will show that the converse holds.

In order to do that, we introduce the notion of a $\delta$-marked abacus corresponding to each partition. Let $\lambda\in\Lambda_{\leq n}$ and choose $b\in\mathbb{N}$ satisfying $b\geq n$. We view $\hat\lambda$ as a $(b+1)$-tuple by adding zeros, and take $\rho_b(\delta)$ to the $(b+1)$-tuple
	\[\rho_b(\delta)=(\delta,-1,-2,\dots,-b).\]
We can then define the $\beta_\delta$-sequence of $\lambda$ to be
	\begin{align*}
		\beta_\delta(\lambda,b)&=\hat\lambda+\rho_b(\delta)+b(\underbrace{1,1,\dots,1}_{b+1})\\
						&=(\delta-|\lambda|+b,\lambda_1-1+b,\lambda_2-2+b,\lambda_3-3+b,\dots,2,1,0).
	\end{align*}
It is clear that $\mu\in \mathcal{O}_\lambda^p(n;\delta)$ if and only if $\beta_\delta(\mu,b)\sim_p\beta_\delta(\lambda,b)$. We then use the  $\beta_\delta$-sequence to construct the \emph{$\delta$-marked abacus} of $\lambda$ as follows:
	\begin{enumerate}
		\item Take an abacus with $p$ runners, labelled $0$ to $p-1$ from left to right. The positions of the abacus start at 0 and increase from left to right, moving down the runners.
		\item Set $v_\lambda$ to be the unique integer $0\leq v_\lambda\leq p-1$ such that $\beta_\delta(\lambda,b)_0=\delta-|\lambda|+b\equiv v_\lambda$ mod $p$. Place a $\vee$ on top of runner $v_\lambda$.
		\item For the rest of the entries of $\beta_\delta(\lambda,b)$, place a bead in the corresponding position of the abacus, so that the final abacus contains $b$ beads, as in Section 2.
	\end{enumerate}

\begin{example}\label{ex:partabacus}
Let $p=5$, $\delta=1$, $\lambda=(2,1)$. We choose an integer $b\geq3$, for instance $b=7$. Then the $\beta$-sequence is
	\begin{align*}
		\beta_\delta(\lambda,7)&=(1-3+7,2-1+7,\dots,0)\\
						&=(5,8,6,4,3,2,1,0).		
	\end{align*}
	The resulting $\delta$-marked abacus is given in Figure \ref{fig:partabacus}.
	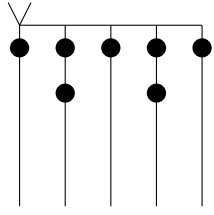
\begin{figure}
		\centering
		\begin{tikzpicture}[scale=0.6]
			\draw (1,4)--(5,4);
			\foreach \x in {1,...,5}
			{\draw (\x,0)--(\x,4);
			\fill[black] (\x,3.5) circle (6pt);}
			\fill[black] (2,2.5) circle (6pt);
			\fill[black] (4,2.5) circle (6pt);
			\draw (0.75,4.5)--(1,4)--(1.25,4.5);
		\end{tikzpicture}
		\caption{The $\delta$-marked abacus of $\lambda$, with $\lambda=(2,1)$, $p=5$, $\delta=1$ and $b=7$.}\label{fig:partabacus}
	\end{figure}
\end{example} 

	Note that if we ignore the $\vee$ we recover James' abacus representing $\lambda$ with $b$ beads explained in Section 2.
If the context is clear, we will use \emph{marked abacus} to mean $\delta$-marked abacus.

Recall the definition of $\Gamma(\lambda,b)$ from \eqref{eq:gammaabacus}. If we now use the marked abacus, we similarly define $\Gamma_\delta(\lambda,b)=(\Gamma_\delta(\lambda,b)_0,\Gamma_\delta(\lambda,b)_1,\dots,\Gamma_\delta(\lambda,b)_{p-1})$ by
	\[\Gamma_\delta(\lambda,b)_i=\begin{cases}
								\Gamma(\lambda,b)_i&\text{if }i\neq v_\lambda\\
								\Gamma(\lambda,b)_i+1&\text{if }i=v_\lambda.
							\end{cases}\]
Now it's easy to see that we have $\mu\in \mathcal{O}^p_\lambda(n;\delta)$ if and only if
				$\Gamma_\delta(\mu,b)=\Gamma_\delta(\lambda,b)$.

We now use the $\delta$-marked abacus to show that set $\mathcal{O}_\lambda^p(n;\delta)$ contains a unique minimal element.

\begin{defn} Let $\lambda\in \Lambda_{\leq n}$.
For  $\mathcal{O}=\mathcal{O}_\lambda^p(n;\delta)$ we define $\lambda_\mathcal{O}$ to be the partition such that
			\begin{itemize}
				\item[(i)] $\Gamma_\delta (\lambda_\mathcal{O},b)=\Gamma_\delta(\lambda,b)$,
				\item[(ii)] All beads on the marked abacus of $\lambda_\mathcal{O}$ are as far up their runners as possible,
				\item[(iii)] The runner $v_{\lambda_\mathcal{O}}$ is the rightmost runner $i$ such that $\Gamma_\delta(\lambda_\mathcal{O},b)_i$ is maximal.
			\end{itemize}
\end{defn}
		The partition $\lambda_\mathcal{O}$ is well defined, i.e. it is independent of the number of beads used. To see this, note that by adding $m$ beads to the abacus we move each existing bead $m$ places to the right. Moreover since $v_{\lambda_\mathcal{O}}\equiv\delta-|\lambda_\mathcal{O}|+b$ mod $p$, we also move the $\vee$ by $m$ places to the right. Therefore none of the beads change their relative positions to one another and the partition $\lambda_\mathcal{O}$ remains unchanged. This is illustrated in Figure \ref{fig:partminwelldef}.
	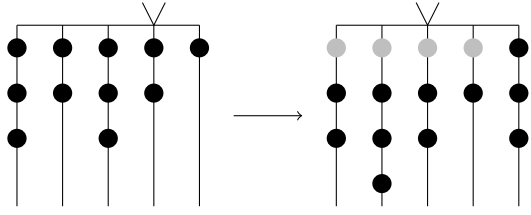
\begin{figure}
		\centering
		\begin{tikzpicture}[scale=0.6]
			\draw (1,4)--(5,4);
			\foreach \x in {1,...,5}
			{\draw (\x,0)--(\x,4);
			\fill[black] (\x,3.5) circle (6pt);}
			\fill[black] (1,2.5) circle (6pt);
			\fill[black] (2,2.5) circle (6pt);
			\fill[black] (3,2.5) circle (6pt);
			\fill[black] (4,2.5) circle (6pt);
			\fill[black] (3,1.5) circle (6pt);
			\fill[black] (1,1.5) circle (6pt);
			\draw (3.75,4.5)--(4,4)--(4.25,4.5);
			
			\draw[->] (5.75,2)--(7.25,2);
			
			\draw (1+7,4)--(5+7,4);
			\foreach \x in {1,...,5}
			{\draw (\x+7,0)--(\x+7,4);
			\fill[black] (\x+7,2.5) circle (6pt);}
			\foreach \x in {1,...,4}
			\fill[lightgray] (\x+7,3.5) circle (6pt);
			\fill[black] (5+7,3.5) circle (6pt);
			\fill[black] (1+7,1.5) circle (6pt);
			\fill[black] (2+7,1.5) circle (6pt);
			\fill[black] (3+7,1.5) circle (6pt);
			\fill[black] (5+7,1.5) circle (6pt);
			\fill[black] (2+7,0.5) circle (6pt);
			\draw (0.75+9,4.5)--(1+9,4)--(1.25+9,4.5);
		\end{tikzpicture}
		\caption{Adding 4 beads (coloured grey) to the abacus of $\lambda_\mathcal{O}=(2,1)$. Each existing bead (coloured black) and the $\vee$ is moved 4 places to the right.}\label{fig:partminwelldef}
	\end{figure}

\begin{prop}\label{prop:partminimal}
		Let $\lambda\in\Lambda_{\leq n}$, then the set $\mathcal{O}=\mathcal{O}_\lambda^p(n;\delta)$ contains a unique element  of minimal degree, namely $\lambda_{\mathcal{O}}$.  More precisely, if $\mu\in\mathcal{O}$ then $|\mu|\geq|\lambda_\mathcal{O}|$, with equality if and only if $\mu=\lambda_\mathcal{O}$.
\end{prop}
	\begin{proof}
 Let $\mu\in \mathcal{O}_\lambda^p(n;\delta)$ with $\mu\neq \lambda_{\mathcal{O}}$. Since
		$\Gamma_\delta(\mu,b)=\Gamma_\delta(\lambda,b)$, it's easy to see that there is a sequence of partitions in $\mathcal{O}$  $$\mu=\eta^{(0)} , \eta^{(1)}, \ldots , \eta^{(t)} = \lambda_{\mathcal{O}}$$ for some $t>0$ such that for each $0\leq i\leq t-1$ the partitions $\eta=\eta^{(i)}$ and $\eta'=\eta^{(i+1)}$ are related in precisely one of the following ways.

\textbf{Case 1} The partition $\eta$ is not a $p$-core and the marked abacus of $\eta'$ is obtained from the marked abacus of $\eta$ by pushing a bead one step up its runner. In this case we have $|\eta'|=|\eta|-p$ and so $v_{\eta'}=v_\eta$ and $\eta'\in \mathcal{O}_\eta^k(n;\delta)$.

\textbf{Case 2} The partition $\eta$ is a $p$-core. Then as $\eta\neq \lambda_{\mathcal{O}}$, it does not satisfy condition (iii) above. Now pick the first runner, say runner $j$, to the right of $v_{\eta}$ satisfying $\Gamma_\delta(\eta, b)_{v_{\eta}}\leq \Gamma_\delta(\eta,b)_{j}$. Then the marked abacus of $\eta'$ is obtained from that of $\eta$ by moving the lowest bead on runner $j$ exactly $j-v_{\eta}$ steps to the left to runner $v_{\eta}$.   In this case we have $|\eta'|=|\eta|-(j-v_{\eta})$ and so we have $v_{\eta'}=j$ and $\eta'\in \mathcal{O}_\eta^k(n;\delta)$. This is illustrated in Figure \ref{fig:partminimal}.

In both cases we saw that $|\eta'|<|\eta|$ and so we get $|\lambda_{\mathcal{O}}|<\mu$ as required.
	\begin{figure}
		\centering
		\begin{tikzpicture}[scale=0.6]
			\draw (1,4)--(5,4);
			\foreach \x in {1,...,5}
			{\draw (\x,0)--(\x,4);
			\fill[black] (\x,3.5) circle (6pt);}
			\fill[black] (1,2.5) circle (6pt);
			\fill[black] (2,2.5) circle (6pt);
			\fill[black] (3,2.5) circle (6pt);
			\fill[black] (4,2.5) circle (6pt);
			\fill[black] (3,1.5) circle (6pt);
			\fill[black] (4,1.5) circle (6pt);
			\draw (0.75,4.5)--(1,4)--(1.25,4.5);
			\draw (3,-0.7) node {$\eta$};
			
			\draw[->] (5.75,2)--(7.25,2);
			
			\draw (1+7,4)--(5+7,4);
			\foreach \x in {1,...,5}
			{\draw (\x+7,0)--(\x+7,4);
			\fill[black] (\x+7,3.5) circle (6pt);}
			\fill[black] (1+7,2.5) circle (6pt);
			\fill[black] (2+7,2.5) circle (6pt);
			\fill[black] (3+7,2.5) circle (6pt);
			\fill[black] (4+7,2.5) circle (6pt);
			\fill[black] (1+7,1.5) circle (6pt);
			\fill[black] (4+7,1.5) circle (6pt);
			\draw (0.75+9,4.5)--(1+9,4)--(1.25+9,4.5);
			\draw[->] (9.75,1.5) .. controls (9.5,1.25) and (8.5,1.25) .. (8.25,1.5);
			\draw (3+7,-0.7) node {$\eta'$};

		\end{tikzpicture}
		\caption{Constructing the marked abacus of $\eta'$ from $\eta$ in Case 2.}\label{fig:partminimal}
	\end{figure}
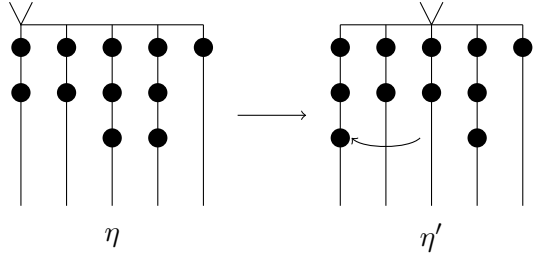
	\end{proof}

The aim now is to show that any partition $\mu\in \mathcal{O}$ is in the same cell-block as $\lambda_\mathcal{O}$. We first need the following proposition to relate cell-blocks over a field of characteristic zero to those over a field of positive characteristic.

\begin{prop}\label{thm:partreduceblocks} Take $\delta\in \mathbb{Z}$ with $0<\delta \leq p-1$ and let $\lambda,\mu\in \Lambda_{\leq n}$.
		If $\mu\in\mathcal{B}_\lambda(P^K_n(\delta+rp))$ for some $r\in\mathbb{Z}$, then $\mu\in\mathcal{B}_\lambda(P_n^k(\delta))$.
		
	\end{prop}
		\begin{proof}

			By  cellularity, the cell-blocks of the partition algebras are the classes of the equivalence relation on $\Lambda_{\leq n}$ generated by non-zero decomposition numbers. So we only need to show that if 
\begin{equation}\label{nonzerodecompK}
[\Delta_\mu^K(n;\delta +rp): L_\lambda^K(n;\delta+rp)]\neq 0
\end{equation}
 then $\mu\in \mathcal{B}_\lambda(P^k_n(\delta))$.
Now note that (\ref{nonzerodecompK}) implies that there exists a submodule $M\subset \Delta_\mu^K(n;\delta +rp)$ and a non-zero homomorphism
$$
\psi \, : \, \Delta_{\lambda}^K(n;\delta+rp) \longrightarrow \Delta_{\mu}^K(n;\delta+rp)/M. 
$$
Now as $\delta+rp=\delta$ in $k$ and the cell modules have  $R$-forms, we have  $k\otimes_R \Delta_{\lambda}^R(n, \delta + rp)=\Delta_{\lambda}^k(n, \delta)$ (and similarly for $\mu$). Now as all modules here have finite rank, the homomorphism $\psi$ can be rescaled if necessary so that $k\otimes_R\psi(\Delta_\lambda^R(n;\delta+rp))$ is non zero and \newline $k\otimes_R \psi(\Delta_\lambda^R(n;\delta +rp)) \subseteq (k\otimes_R \Delta_\mu^R(n;\delta +rp))/(k\otimes_R N) = \Delta_\mu^k(n;\delta)/(k\otimes_R N)$ for some $R$-form $N$ of $M$. This gives a non-zero homomorphism
$$\bar{\psi}\, :  \Delta_{\lambda}^k(n;\delta) \longrightarrow \Delta_{\mu}^k(n;\delta)/(k\otimes_R N).$$
This shows that $\mu\in \mathcal{B}_\lambda(P^k_n(\delta))$ as required. 
		\end{proof}
		
We now set $b=n$, so that all marked abaci have $n$ beads.
\begin{prop}\label{prop:partminblocks}
	Let $\lambda\in\Lambda_{\leq n}$ and write $\mathcal{O}=\mathcal{O}_\lambda(n;\delta)$. If $\lambda\neq\lambda_\mathcal{O}$, i.e. $\lambda$ is not minimal in its orbit, then there exists a partition $\mu\in\mathcal{O}$ with $|\mu|<|\lambda|$ and $\mu\in\mathcal{B}_\lambda(P^k_n(\delta))$.
\end{prop}
	\begin{proof}
		If $\lambda\neq\lambda_\mathcal{O}$, then as in the proof of Proposition \ref{prop:partminimal} either $\lambda$ is not a $p$-core or $v_\lambda$ is not the rightmost runner $i$ such that $\Gamma_\delta(\lambda,n)_i$ is maximal. We now refine the cases provided in the proof of Proposition \ref{prop:partminimal} to construct a partition $\mu$ with the required properties.

		\textbf{Case A}~ The partition $\lambda$ is not a $p$-core and there is a bead, say the $j$-th bead, which lies on runner $v_\lambda$ and has an empty space immediately above it.
Let $\mu$ be the partition obtained by moving the $j$-th bead one space up its runner (as illustrated in Figure \ref{fig:partcase1}), so that it now occupies position $\lambda_j-j+n-p$. Note that $|\mu|=|\lambda|-p$. In particular, since no beads are changing runners we get  $\Gamma_\delta(\mu,n)=\Gamma_\delta(\lambda,n)$ and so $\mu\in\mathcal{O}$. Note that setwise the $\beta_\delta$-sequence $\beta_\delta(\mu,n)$ must be
		\begin{equation}
			(\delta-|\lambda|+p+n,\lambda_1-1+n,\dots,\lambda_j-j-p+n,\dots,0)\label{eq:partbetacase1}
		\end{equation}
		since no other beads move. However since bead $j$ lies on runner $v_\lambda$, we can find $r\in\mathbb{Z}$ such that
			\[\delta-|\lambda|+n+(r+1)p=\lambda_j-j+n\]
		and can therefore rewrite \eqref{eq:partbetacase1} as
			\[(\lambda_j-j+n-rp,\lambda_1-1+n,\dots,\delta-|\lambda|+n+rp,\dots,0).\]
		Thus, for an appropriate element $w\in\langle s_{i,j}:1\leq i<j\leq n\rangle =W_n$, we have:
		\[w^{-1}(\beta_\delta(\mu,n))=(\lambda_j-j-rp+n,\lambda_1-1+n,\dots,\underbrace{\delta-|\lambda|+rp+n}_{j\text{-th place}},\dots,0)\]
		and hence
		\begin{align*}
			\beta_\delta(\lambda,n)-w^{-1}(\beta_\delta(\mu,n))&=(\delta-|\lambda|-\lambda_j+j+rp)(\varepsilon_0-\varepsilon_j)\\
			&=\langle\hat\lambda+\rho_n(\delta+rp),\varepsilon_0-\varepsilon_j\rangle(\varepsilon_0-\varepsilon_j).
		\end{align*}
	We can rewrite this as
		\[w^{-1}(\hat\mu+\rho_n(\delta)+n(1,\dots,1))=\hat\lambda+\rho_n(\delta)+n(1,\dots,1)-\langle\hat\lambda+\rho_n(\delta+rp),\varepsilon_0-\varepsilon_j\rangle(\varepsilon_0-\varepsilon_j).\]
		Since $W_n$ does not act on the $0$-th postition, both the elements $(rp,0,0,\dots,0)$ and $n(1,1,\dots,1)$ are unchanged by $w^{-1}$. Thus:
		\begin{align*}
			&w^{-1}(\hat\mu+\rho_n(\delta)+(rp,0,\dots,0))\\&~~~~~~=\hat\lambda+\rho_n(\delta)+(rp,0,\dots,0)-\langle\hat\lambda+\rho_n(\delta+rp),\varepsilon_0-\varepsilon_j\rangle(\varepsilon_0-\varepsilon_j)\\
			\implies&w^{-1}(\hat\mu+\rho_n(\delta+rp))=s_{0,j}(\hat\lambda+\rho_n(\delta+rp))\\
			\implies&\hat\mu=ws_{0,j}(\hat\lambda+\rho_n(\delta+rp))-\rho_n(\delta+rp)\\
			\implies&\hat\mu=ws_{0,j}\cdot_{\delta+rp}\hat\lambda.
		\end{align*}
	Therefore $\hat\mu\in \widehat{W}_n\cdot_{\delta+rp}\hat\lambda$, and so, by Theorem \ref{thm:part0blocksgeom}, $\mu\in\mathcal{B}_\lambda(P^K_n(\delta+rp))$. Proposition \ref{thm:partreduceblocks} then provides the final result.

	\begin{figure}
		\centering
		\begin{tikzpicture}[scale=0.6]
			\draw (1,4)--(5,4);
			\foreach \x in {1,...,5}
			\draw (\x,1)--(\x,4);
			\fill[black] (3,2.5) circle (6pt);
			\draw (0.75+2,4.5)--(1+2,4)--(1.25+2,4.5);
			
			\draw[->] (5.75,2.5)--(7.25,2.5);
			
			\draw (1+7,4)--(5+7,4);
			\foreach \x in {1,...,5}
			\draw (\x+7,1)--(\x+7,4);
			\fill[black] (3+7,3.5) circle (6pt);
			\draw (0.75+9,4.5)--(1+9,4)--(1.25+9,4.5);
		\end{tikzpicture}
		\caption{The movement of beads in Case A.}\label{fig:partcase1}
	\end{figure}
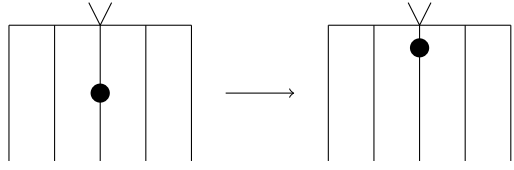

	\textbf{Case B}~~ The partition $\lambda$ is not a $p$-core and the runner $v_\lambda$ is empty. As $\lambda$ is not a $p$-core, there is a bead, say the $j$-th bead (not on the runner $v_\lambda$) with an empty space immediately above it.  Then by moving the $j$-th bead one space up its runner and then across to runner $v_\lambda$, we obtain the abacus of a new partition $\mu$ with $|\mu|=|\lambda|-m$ for some $m>0$ (as illustrated in Figure \ref{fig:partcase2}). Since bead $j$ is now on runner $v_\lambda$ and occupies position $\lambda_j-j+n-m$, we see that 
		\begin{equation}
		\lambda_j-j+n-m=\delta-|\lambda|+n+rp\label{eq:partcase2}
		\end{equation}
		for some $r\in\mathbb{Z}$. The runner $v_\mu$ is given by
			\begin{align*}
				\delta-|\mu|+n&=\delta-|\lambda|+m+n\\
						&=\lambda_j-j+n-rp.
			\end{align*}
		So runner $v_\mu$ is equal to the runner previously occupied by bead $j$. Therefore $\Gamma_\delta(\mu,n)=\Gamma_\delta(\lambda,n)$, and so $\mu\in\mathcal{O}_\lambda^p(n;\delta)$. We also have that setwise the $\beta_\delta$-sequence $\beta_\delta(\mu,n)$ is
		\[(\delta-|\lambda|+m+n,\lambda_1-1+n,\dots,\lambda_j-j-m+n,\dots,0)\]
		which, by using \eqref{eq:partcase2}, we may rewrite as
		\[(\lambda_j-j+n-rp,\lambda_1-1+n,\dots,\delta-|\lambda|+n+rp,\dots,0).\]
		We can then continue as in Case A to deduce that $\mu\in \mathcal{B}_\lambda(P^k_n(\delta))$.
	\begin{figure}
		\centering
		\begin{tikzpicture}[scale=0.6]
			\draw (1,4)--(5,4);
			\foreach \x in {1,...,5}
			\draw (\x,1)--(\x,4);
			\fill[black] (4,2.5) circle (6pt);
			\draw (0.75+1,4.5)--(1+1,4)--(1.25+1,4.5);
			
			\draw[->] (5.75,2.5)--(7.25,2.5);
			
			\draw (1+7,4)--(5+7,4);
			\foreach \x in {1,...,5}
			\draw (\x+7,1)--(\x+7,4);
			\fill[black] (3+6,3.5) circle (6pt);
			\draw (0.75+10,4.5)--(1+10,4)--(1.25+10,4.5);
		\end{tikzpicture}
		\caption{The movement of beads in Case B.}\label{fig:partcase2}
	\end{figure}
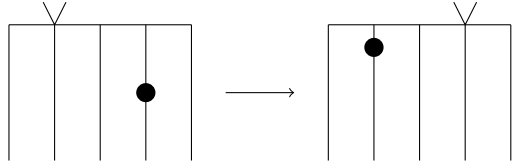

	\textbf{Case C}~ The partition $\lambda$ is not a $p$-core, there is at least one bead on runner $v_\lambda$ and all the beads on runner $v_\lambda$ are as far up as possible. Then there is a bead, say the $j$-th bead, not on the runner $v_\lambda$, with an empty space immediately above it. Define $\nu$ to be the partition obtained by moving the $j$-th bead one space up and moving the last bead on runner $v_\lambda$ one space down.  By Theorems \ref{thm:nakayama} and \ref{thm:decompsamesize} we have that $\nu\in \mathcal{B}_\lambda(P^k_n(\delta))$. Now note that $|\nu|=|\lambda|$ and $\nu$ satisfies the conditions of Case A, and using this case we can find $\mu\in \mathcal{B}_\nu(P^k_n(\delta))=\mathcal{B}_\lambda(P^k_n(\delta))$ with $|\mu|<|\nu|=|\lambda|$.  This is illustrated in Figure \ref{fig:partcase3}.
	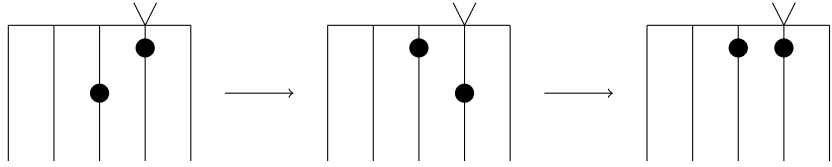
\begin{figure}
		\centering
		\begin{tikzpicture}[scale=0.6]
			\draw (1,4)--(5,4);
			\foreach \x in {1,...,5}
			\draw (\x,1)--(\x,4);
			\fill[black] (4,3.5) circle (6pt);
			\fill[black] (3,2.5) circle (6pt);
			\draw (0.75+3,4.5)--(1+3,4)--(1.25+3,4.5);
			
			\draw[->] (5.75,2.5)--(7.25,2.5);
			
			\draw (1+7,4)--(5+7,4);
			\foreach \x in {1,...,5}
			\draw (\x+7,1)--(\x+7,4);
			\fill[black] (3+7,3.5) circle (6pt);
			\fill[black] (4+7,2.5) circle (6pt);
			\draw (0.75+10,4.5)--(1+10,4)--(1.25+10,4.5);
			
			\draw[->] (5.75+7,2.5)--(7.25+7,2.5);
			
			\draw (1+14,4)--(5+14,4);
			\foreach \x in {1,...,5}
			\draw (\x+14,1)--(\x+14,4);
			\fill[black] (3+14,3.5) circle (6pt);
			\fill[black] (4+14,3.5) circle (6pt);
			\draw (0.75+17,4.5)--(1+17,4)--(1.25+17,4.5);

		\end{tikzpicture}
		\caption{The movement of beads in Case C.}\label{fig:partcase3}
	\end{figure}

\textbf{Case D}~ Now suppose that $\lambda$ is a $p$-core, but $v_\lambda$ is not the rightmost runner $i$ such that $\Gamma_\delta(\lambda,n)_i$ is maximal. As in Case 2 in the proof of Proposition \ref{prop:partminimal} we can pick the first runner, say runner $i$, to the right of $v_{\lambda}$ satisfying $\Gamma_\delta(\lambda, b)_{v_{\lambda}}\leq \Gamma_\delta(\lambda,b)_{i}$ and define $\mu$ to be the partition obtained  by moving the lowest bead on runner $i$, say it is the $j$-th bead, exactly $i-v_{\lambda}$ steps to the left to runner $v_{\lambda}$ (as illustrated in Figure \ref{fig:partminimal}). Then we have $|\mu|=|\lambda|-(i-v_{\lambda})$ and so $v_{\mu}=i$. 
Now we have
$$\hat{\mu}=(-|\lambda|+(i-v_{\lambda}), \lambda_1, \ldots , \lambda_{j-1}, \lambda_j-(i-v_\lambda), \lambda_{j+1}, \ldots 0).$$ 
As $\lambda_j-j+n\equiv i \, {\rm mod}\, p$ and $-|\lambda|+\delta +n\equiv v_{\lambda} \, {\rm mod}\, p$ we get that 
$$\lambda_j-j-i = -|\lambda|+\delta -v_\lambda +rp$$
for some $r\in \mathbb{Z}$. This gives $-|\lambda|+(i-v_\lambda)=\lambda_j-j-(\delta+rp)$ and $\lambda_j-(i-v_\lambda)=-|\lambda|+(\delta +rp) +j$. Thus we have
\begin{eqnarray*}
\hat{\mu} &=& (\lambda_j-j-(\delta+rp), \lambda_1, \ldots , \lambda_{j-1}, -|\lambda|+(\delta +rp) +j, \lambda_{j+1}, \ldots , 0)\\
&=& s_{0,j}\cdot_{\delta +rp} \hat{\lambda}.
\end{eqnarray*}
Using Theorem \ref{thm:part0blocks} and Proposition \ref{thm:partreduceblocks} we deduce that $\mu\in \mathcal{B}_\lambda(P^k_n(\delta))$.
\end{proof}

	We immediately deduce the following:
\begin{cor}\label{cor:blocks>orbits}
	Let $\lambda,\mu\in\Lambda_{\leq n}$. If $\mu\in\mathcal{O}_\lambda^p(n;\delta)$, then $\mu\in\mathcal{B}_\lambda(P^k_n(\delta))$.
\end{cor}
	\begin{proof}
		Since each set $\mathcal{O}=\mathcal{O}_\lambda^p(n;\delta)$ contains a unique minimal element $\lambda_\mathcal{O}$, it suffices to show that $\lambda_\mathcal{O}\in\mathcal{B}_{\lambda}(P^k_n(\delta))$. We prove this by induction on $|\lambda|$.
		
		If $|\lambda|$ is minimal, then $\lambda=\lambda_\mathcal{O}$ and there is nothing to prove. So suppose otherwise, i.e. that $\lambda\neq\lambda_\mathcal{O}$. Then by Proposition \ref{prop:partminblocks} there is a partition $\nu\in\mathcal{O}$ with $|\nu|<|\lambda|$ and $\nu\in\mathcal{B}_\lambda(P^k_n(\delta))$. By our inductive step, we then have $\lambda_\mathcal{O}\in\mathcal{B}_\nu(P^k_n(\delta))$. But cell-blocks are either disjoint or coincide entirely, so $\mathcal{B}_\nu(P^k_n(\delta))=\mathcal{B}_\lambda(P^k_n(\delta))$ and $\lambda_\mathcal{O}\in\mathcal{B}_\lambda(P^k_n(\delta))$ as required.
	\end{proof}
	
\noindent We can now combine Corollaries \ref{cor:blocks<orbits} and \ref{cor:blocks>orbits} to obtain the main result of this paper.
	
\begin{thm}\label{thm:partpblocks} Assume $\delta\in \bbf_p$ with $\delta \neq 0$.  For any $\lambda\in\Lambda_{\leq n}$ we have  $\mathcal{B}_\lambda(P^k_n(\delta))=\mathcal{O}_\lambda^p(n;\delta)$. 
\end{thm}

\subsection*{Reflection geometry.}

We now give a reformulation of Theorem \ref{thm:partpblocks} in terms of the action of an affine reflection group, extending the ones given for the blocks of the symmetric group (over $k$) and for the partition algebra (over $K$) in Sections 3 and 5 respectively. In fact, the affine reflection group we need to consider here, which we will denote by $\widehat{W}_n^p$  is precisely the group generated by $W_n^p$ and $\widehat{W}_n$. Using Proposition \ref{thm:partreduceblocks} and Theorem  \ref{thm:decompsamesize}, we know that our new group should certainly contain $W_n^p$ and $\widehat{W}_n$ but it is perhaps surprising that it is in fact the smallest such group.

 Let $\widehat{W}_n^p$ be the affine reflection group on $\widehat{E}_n$  generated by the affine reflections $s_{i,j,rp}$ $(0\leq i<j\leq n),r\in\mathbb{Z}$, where $s_{i,j,rp}(x)=x-\big(\langle x,\varepsilon_i-\varepsilon_j\rangle-rp\big)(\varepsilon_i-\varepsilon_j)$ for all $x\in \widehat{E}_n$.  Using the same embedding $\hat\lambda\in \widehat{E}_n$ of a partition $\lambda\in \Lambda_{\leq n}$ and the same shifted action by $\rho_n(\delta)$ as before we can reformulate Theorem \ref{thm:partpblocks} as follows.
\begin{thm}
 Assume $\delta\in \bbf_p$ with $\delta \neq 0$.	Let $\lambda, \mu\in \Lambda_{\leq n}$. Then we have $\mu\in \mathcal{B}_\lambda(P_n^k(\delta))$ if and only if $\hat{\mu}\in \widehat{W}_n^p \cdot_{\delta} \hat{\lambda}$.
\end{thm}
	\begin{proof}
		We have $\hat\mu\in \widehat{W}_n^p\cdot\hat\lambda$ if and only if there is some $w\in \widehat{W}_n$ and $\alpha\in\mathbb{Z}\widehat{\Phi}_n$ such that
			\[\hat\mu+\rho_n(\delta)=w(\hat\lambda+\rho_n(\delta))+p\alpha.\]
		Conversely, we have $\hat\mu+\rho_n(\delta)\sim_p\hat\lambda+\rho_n(\delta)$ if and only if
			\[\hat\mu+\rho_n(\delta)=w(\hat\lambda+\rho_n(\delta))+px\]
		for some $w\in \widehat{W}_n$ and $x\in\mathbb{Z}^{n+1}$. But as $\sum_{i=0}^n(\hat\mu)_i=\sum_{i=0}^n(\hat\lambda)_i=0$ we see that $\sum_{i=0}^nx_i=0$, and therefore $x\in\mathbb{Z}\widehat{\Phi}_n$.
	\end{proof}
					
\section{Limiting blocks}

We continue to assume that $\delta\in \bbf_p$ with $\delta\neq 0$ throughout this final section.

The proof of Proposition \ref{prop:partminblocks} uses two main ingredients to show that a partition $\lambda$ is in the same block as the minimal element in its orbit $\lambda_{\mathcal{O}}$, namely:
\begin{enumerate}
\item the blocks of the partition algebra $P_n^K(\delta +rp)$ for all $r\in \mathbb{Z}$ (used in Cases A, B and D), and
\item the blocks of the symmetric group algebras $k\sym_{n-t}$ for all $0\leq t\leq n$ (used in Case C).
\end{enumerate}

One might wonder whether we can give a proof of this result which only uses (1). This would give a proof of the modular blocks of the partition algebra with integer parameter without assuming any result about the modular representation theory of the symmetric group (note that Section 6 giving the necessary condition for the blocks does not use the modular representation theory of the symmetric group). This is not possible as the following example shows.					

\begin{example}\label{ex:partlimblocks}
Consider the partition $\lambda = (7, 3^2, 1^2)$ with $p=5$, $n=15$ and $\delta=1$. Then the marked abacus of $\lambda$ with $15$ beads is illustrated in Figure \ref{fig:partabacussym}.
	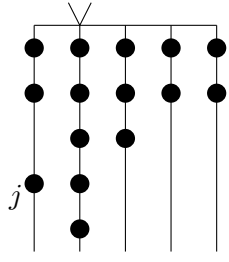
\begin{figure}
		\centering
		\begin{tikzpicture}[scale=0.6]
			\draw (1,5)--(5,5);
			\foreach \x in {1,...,5}
			\draw (\x,0)--(\x,5);
			\fill[black] (1,4.5) circle (6pt);
\fill[black] (1,3.5) circle (6pt); \fill[black] (1,1.5) circle (6pt); 
			\fill[black] (2,4.5) circle (6pt); \fill[black] (2,3.5) circle (6pt);
\fill[black] (2,2.5) circle (6pt); \fill[black] (2,1.5) circle (6pt); \fill[black] (2,0.5) circle (6pt);
			\fill[black] (3,4.5) circle (6pt);
			\fill[black] (3,3.5) circle (6pt);
			\fill[black] (3,2.5) circle (6pt);\draw (0.6,1.2) node {$j$};
			\fill[black] (4,4.5) circle (6pt);
			\fill[black] (4,3.5) circle (6pt); \fill[black] (5,4.5) circle (6pt);
			\fill[black] (5,3.5) circle (6pt); 
			\draw (1.75,5.5)--(2,5)--(2.25,5.5);
		\end{tikzpicture}
		\caption{The marked abacus of $\lambda=(7,3^2,1^2)$ with $p=5$, $\delta=1$ and $b=15$.}\label{fig:partabacussym}
	\end{figure}
From the abacus we see that 
$$\mathcal{O}^5_\lambda (15;1)=\{(12,3), (7,4^2), (7,3^2,1^1), (7,3,2,1^3), (7,3,1^5), (7,3)\}.$$
Now the only $\delta+rp$-pair, for any $r\in \mathbb{Z}$, among these partitions is given by
$$(7,3) \hookrightarrow_{21} (12,3).$$
Thus in this case it is impossible to show that $\lambda$ is in the same block as $\lambda_{\mathcal{O}}=(7,3)$ without using the blocks of the symmetric group algebra $k\sym_{15}$.
\end{example}

However, we will show that if we allow ourselves to increase the degree of the partitions we consider, then it is possible to link any partition $\lambda$ to the minimal element $\lambda_\mathcal{O}$ in its orbit using only (the modular reduction of) the blocks of the partition algebras $P^K_n(\delta +rp)$ (for all $r\in \mathbb{Z}$). We make this more precise by defining the notion of \emph{limiting blocks}.

Recall that the globalisation functor $G_{n,n+1}$ defined in Section 4 gives a full embedding of the category $P_n^k(\delta)\text{-\catmod}$ into $P_{n+1}^k(\delta)\text{-\catmod}$. Moreover, we have $G_{n,n+1}(\Delta_\lambda^k(n))=\Delta_\lambda^k(n+1)$. So under this embedding the labelling sets for cell modules correspond via the natural inclusion
$\Lambda_{\leq n} \subset \Lambda_{\leq n+1}$.
We define
$$\Lambda = \bigcup_{n\geq 0}\Lambda_{\leq n}.$$
Now for $\lambda\in \Lambda_{\leq n}$ the functor $G_{n,n+1}$ gives an embedding $\mathcal{B}_{\lambda}(P^k_n(\delta))\subseteq \mathcal{B}_{\lambda}(P^k_{n+1}(\delta))$ and so we can define the \emph{limiting cell-block} containing $\lambda\in \Lambda_{\leq n}$ by
$$\mathcal{B}_\lambda^k(\delta):=\bigcup_{m\geq n} \mathcal{B}_\lambda(P^k_m(\delta)) = \{\mu\in \Lambda \, : \,  \mu\in \mathcal{B}_\lambda(P^k_m(\delta)) \, \mbox{for some $m$}\}.$$
Now we consider $\widehat{E}_n\subset \widehat{E}_{n+1}$ by taking the last coordinate to be zero. Note that for $\lambda, \mu\in \Lambda_{\leq n}$, if $\hat{\mu}+\rho_n(\delta) \sim_p \hat{\lambda}+\rho_n(\delta)$ then $\hat{\mu}+\rho_{n+1}(\delta) \sim_p \hat{\lambda}+\rho_{n+1}(\delta)$. So we also have $\mathcal{O}_\lambda^p(n;\delta) \subseteq \mathcal{O}_\lambda^p(n+1,\delta)$.
Hence we can define
$$\mathcal{O}_\lambda^p(\delta) = \bigcup_{m\geq n} \mathcal{O}_\lambda^p(n;\delta).$$

With these definitions, we prove the following without using any results on the modular representation theory of $\sym_n$.

\begin{thm}\label{thm:partlimblocks}
	Assume $\delta\in \bbf_p$ with $\delta \neq 0$. Let $\lambda\in\Lambda_{\leq n}$, then $\mathcal{B}_\lambda^k(\delta)=\mathcal{O}_\lambda^k(\delta)$. 
\end{thm}
	\begin{proof}
		Suppose $\mu\in\mathcal{B}_\lambda^k(\delta)$. Then $\mu\in\mathcal{B}_\lambda(P^k_n(\delta))$ for some $n\in\mathbb{N}$, and so by Corollary \ref{cor:blocks<orbits} we have $\mu\in\mathcal{O}_\lambda^p(n;\delta)\subset\mathcal{O}_\lambda^p(\delta)$. Suppose now that $\mu\in\mathcal{O}_\lambda^p(\delta)$. Again we have $\mu\in\mathcal{O}_\lambda^p(n;\delta)$ for some $n\in\mathbb{N}$. We can follow the proof of Proposition \ref{prop:partminblocks} but replace Case C with the following alternative:
		\\\\
		\textbf{Case C'}~~ The partition $\lambda$ is not a $p$-core, the runner $v_\lambda$ is not empty and all the beads on runner $v_\lambda$ are as far up as possible. Then there is a bead, say the $j$-th bead , which does not lie on runner $v_\lambda$ with a space immediately above it.
Now consider the partition $\mu$ obtained by moving bead $j$ into the first empty space of runner $v_\lambda$. We then have $|\mu|=|\lambda|+m$ for some $m\in\mathbb{Z}$, and as in \eqref{eq:partcase2} we have
			\begin{equation}
				\lambda_j-j+n+m=\delta-|\lambda|+n+rp\label{eq:partcase3'}
			\end{equation}
		for some $r\in\mathbb{Z}$. The runner $v_\mu$ is given by
			\begin{align*}
				\delta-|\mu|+n&=\delta-|\lambda|-m+n\\
						&=\lambda_j-j+n-rp.
			\end{align*}
		So runner $v_\mu$ is equal to the runner previously occupied by bead $j$ (see Figure \ref{fig:partcase3'}). Therefore $\Gamma_\delta(\mu,n)=\Gamma_\delta(\lambda,n)$, but if $m>0$ then $|\mu|>|\lambda|$ so we may not have $\mu\in\mathcal{O}_\lambda^p(n;\delta)$. However it is true that $\hat\mu +\rho_n(\delta) \sim_p \hat\lambda + \rho_n(\delta)$, so $\mu\in\mathcal{O}_\lambda^p(\delta)$. We also have that setwise the $\beta_\delta$-sequence $\beta_\delta(\mu,n)$ is
		\[(\delta-|\lambda|-m+n,\lambda_1-1+n,\dots,\lambda_j-j+m+n,\dots,0)\]
		which, by using \eqref{eq:partcase3'}, we may rewrite as
		\[(\lambda_j-j+n-rp,\lambda_1-1+n,\dots,\delta-|\lambda|+n+rp,\dots,0).\]
		Arguing as in Case A of Proposition \ref{prop:partminblocks} we see that there is some $w\in W_n$  such that
			\[\hat\mu=ws_{0,j}\cdot_{\delta-rp}\hat\lambda\]
		and so by Theorem \ref{thm:part0blocks} and Proposition \ref{thm:partreduceblocks}, $\mu\in\mathcal{B}_\lambda^k(\delta)$.
		
		Now that bead $j$ occupies the lowest position on runner $v_\lambda$, let bead $j'$ be the bead immediately above this. Let $\nu$ be the partition obtained by moving bead $j'$ into the space above the position previously occupied by bead $j$, i.e. into position $\lambda_j-j+n-p$. Then $|\nu|=|\mu|+m$ for some $m\in\mathbb{Z}$, and we repeat the argument of the previous paragraph to see that $\nu\in\mathcal{B}_\mu^k(\delta)=\mathcal{B}_\lambda^k(\delta)$.
		
		We repeat this process for each bead that falls into Case C of Proposition \ref{prop:partminblocks}. Then we are left only with beads in Case A, B,  or D, and may therefore continue with the proof of Proposition \ref{prop:partminblocks} to get the result.
		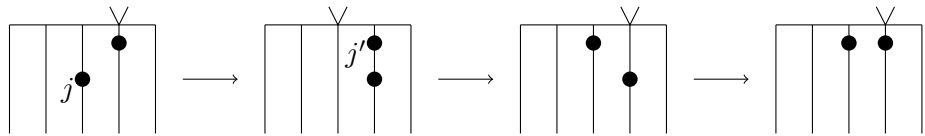
\begin{figure}
		\centering
		\begin{tikzpicture}[scale=0.48]
			\draw (1,4)--(5,4);
			\foreach \x in {1,...,5}
			\draw (\x,1)--(\x,4);
			\fill[black] (4,3.5) circle (6pt);
			\fill[black] (3,2.5) circle (6pt); \draw (2.6,2.2) node {$j$};
			\draw (0.75+3,4.5)--(1+3,4)--(1.25+3,4.5);
			
			\draw[->] (5.75,2.5)--(7.25,2.5);
			
			\draw (1+7,4)--(5+7,4);
			\foreach \x in {1,...,5}
			\draw (\x+7,1)--(\x+7,4);
			\fill[black] (4+7,3.5) circle (6pt);
			\fill[black] (3+8,2.5) circle (6pt);\draw (3.5+7,3.2) node {$j'$};
			\draw (0.75+9,4.5)--(1+9,4)--(1.25+9,4.5);
			
			\draw[->] (5.75+7,2.5)--(7.25+7,2.5);

			\draw (1+14,4)--(5+14,4);
			\foreach \x in {1,...,5}
			\draw (\x+14,1)--(\x+14,4);
			\fill[black] (3+14,3.5) circle (6pt);
			\fill[black] (4+14,2.5) circle (6pt);
			\draw (0.75+17,4.5)--(1+17,4)--(1.25+17,4.5);
			
			\draw[->] (5.75+14,2.5)--(7.25+14,2.5);
			
			\draw (1+21,4)--(5+21,4);
			\foreach \x in {1,...,5}
			\draw (\x+21,1)--(\x+21,4);
			\fill[black] (3+21,3.5) circle (6pt);
			\fill[black] (4+21,3.5) circle (6pt);
			\draw (0.75+24,4.5)--(1+24,4)--(1.25+24,4.5);

		\end{tikzpicture}
		\caption{The movement of beads in Case C'.}\label{fig:partcase3'}
	\end{figure}
	\end{proof}

\bibliographystyle{amsalpha}
\bibliography{thesis}
\end{document}